\documentclass[11pt,reqno]{amsart}
\usepackage{geometry}
\usepackage{graphicx} % Required for inserting images
\usepackage{amsmath,amsfonts,amssymb}
\usepackage{amsthm, mathtools}
\usepackage{hyperref}
\usepackage[capitalise]{cleveref}
\usepackage{comment}
\usepackage{xcolor}
\usepackage[hyperref,doi,url=false,style=alphabetic,maxbibnames=99]{biblatex}

\newcommand{\norm}[1]{\left\lVert#1\right\rVert}

\newcommand{\Z}{\mathbb{Z}}
\newcommand{\R}{\mathbb{R}}
\newcommand{\Q}{\mathbb{Q}}
\newcommand{\M}{\sigma}
\renewcommand{\S}{\mathcal S}
\newcommand{\eps}{\varepsilon}
\newcommand{\hprime}{\widetilde {h}^2}
\renewcommand{\deg}{k}
\newcommand{\ltwo}{\norm{\cdot}_2}
\newcommand{\lone}{\norm{\cdot}_1}

\newcommand{\comm}[1]{{\color{blue}{[#1]}}}

\DeclareMathOperator{\poly}{poly}
\DeclareMathOperator{\polylog}{\poly\log}
\DeclareMathOperator{\vol}{vol}
\DeclareMathOperator{\area}{area}
\DeclareMathOperator{\inj}{inj}
\DeclareMathOperator{\diam}{diam}
\DeclareMathOperator{\Isom}{Isom}
\DeclareMathOperator{\Aut}{Aut}

\newcommand{\sing}{\mathrm{sing}}

\newcommand{\mass}{\mathrm{M}}
\newcommand{\comass}{\mathrm{CM}}
\newcommand{\exact}{\mathrm{exact}}
\newcommand{\coexact}{\mathrm{coexact}}

\DeclareMathOperator{\coker}{coker}
\newcommand{\Ht}[1]{H_1(#1,\Z)_{tors}}

\newtheorem{theorem}{Theorem}[section]
\newtheorem{corollary}{Corollary}[theorem]

\newtheorem{lemma}[theorem]{Lemma}
\newtheorem{prop}[theorem]{Proposition}
\newtheorem{conjecture}[theorem]{Conjecture}
\newtheorem{question}[theorem]{Question}

\theoremstyle{definition}
\newtheorem{definition}[theorem]{Definition}
\newtheorem{construction}[theorem]{Construction}

\theoremstyle{remark}
\newtheorem{remark}[theorem]{Remark}

\title{Expansion and torsion homology of 3-manifolds}
\author{Jonathan Zung}

\begin{document}

\begin{abstract}
    A Riemannian manifold is a called a good rational expander in dimension $i$ if every $i$-cycle  bounds a rational $i+1$-chain of comparatively small volume. We construct 3-manifolds which are good expanders in all dimensions. On the other hand, we show that expanders must be topologically complicated: they must have lots of torsion homology. We also give some applications to topological overlap problems, constructing examples of 3-manifolds with large width over $\R^2$.
\end{abstract}
\maketitle

\begin{section}{Introduction}
Let $M$ be a closed oriented Riemannian 3-manifold. When is $M$ close to being a homology 3-sphere? A homology 3-sphere has vanishing Betti numbers and vanishing torsion homology. One can make quantitative the degree to which $M$ satisfies each of these properties:

\begin{table}[h]
    \centering
    \begin{tabular}{ccc}
$b_0(M)=1$ & $\leadsto$ & Every pair of points is connected by a path of length $<\frac 1 {h_0}$\\
$b_1(M)=0$ & $\leadsto$ & Every loop of length $\ell$ (rationally) bounds a surface of area $<\frac \ell {h_1}$\\
$b_2(M)=0$ & $\leadsto$ & Every surface of area $A$ bounds a region of volume $<\frac A {h_2}$\\
$H_1(M,\Z)_{tors}=0$ & $\leadsto$ & $|H_1(M,\Z)_{tors}|$ is small
    \end{tabular}
    \label{tab:my_label}
\end{table}
The numbers $h_0$, $h_1$ and $h_2$ are called Cheeger constants for $M$. Their study fits into the growing field of high dimensional expansion; see the survey by Lubotzky \cite{lubotzky_high_2017}. For more precise definitions, see \cref{sec:cheeger}.  When $1/h_i$ is small, we say that $M$ has good expansion in dimension $i$.  In this paper, we study the following question:

\begin{question}\label{question:1}
    Does there exist a sequence of Riemannian 3-manifolds $\{M_i\}$ having bounded geometry and $\vol(M_i)\to \infty$, and controlled values of $\frac 1 {h_0(M_i)}$, $\frac 1 {h_1(M_i)}$, $\frac 1 {h_2(M_i)}$, and $|H_1(M_i,\Z)_{tors}|$? For example, could we take all these constants to be of size $o(\vol(M_i)^\varepsilon)$?
\end{question}

$1/h_0$ is simply the diameter of $M$, so it is easy to arrange $1/h_0 \lesssim \log(\vol(M_i))$. It is also easy to construct examples of 3-manifolds with $1/h_2$ uniformly bounded above and $|H_1(M,\Z)_{tors}|=1$; one can simply construct a manifold which coarsely looks like an expander graph. In prior work with Abdurrahman et al, we investigated sequences of hyperbolic 3-manifolds with $H_1(M,\Z)=0$ and $1/h_1$ uniformly bounded above \cite{abdurrahman_hyperbolic_2024}. In the present paper, we construct a sequence of 3-manifolds with simultaneous control on $h_0$, $h_1$ and $h_2$:

\begin{theorem}\label{thm:0}
    There exists a sequence of rational homology 3-spheres $M_i$ equipped with metrics of bounded geometry such that $\vol(M_i)\to \infty$ and $$\max\left(\frac 1 {h_0(M_i)}, \frac 1 {h_1(M_i)}, \frac 1 {h_2(M_i)}\right)\lesssim\poly(\log(\vol(M_i))).$$ In a similar vein, there exists a sequence of rational homology 3-spheres $M_i$ with metrics of bounded geometry, $\vol(M_i) \to \infty$, and the spectral gap for the Hodge Laplacian on all differential forms uniformly bounded below.
\end{theorem}

Our construction starts with a simplicial 2-complex having both good expansion and coexpansion properties. The 3-manifold we build inherits good expansion properties. The two parts of \cref{thm:0} pertain to Cheeger constants with $L^1$ and $L^2$ norms. Lipnowski--Stern, Boulanger--Courtois, and Rudd proved theorems relating $L^1$ and $L^2$ Cheeger constants as well as stable commutator length \cite{lipnowski_geometry_2018, boulanger_cheeger-like_2022, rudd_stable_2021}. Along the way, we will prove a similar Cheeger--Buser inequality (\cref{thm:cheeger}) which simplifies some of their results and generalizes to all $i$-forms. This addresses a question of Boulanger--Courtois \cite[Section 6, Question 1]{boulanger_cheeger-like_2022}. The second part of \cref{thm:0} is a response to a question of Sarnak, who asked for a sequence of \emph{hyperbolic} 3-manifolds with uniform spectral gap for $i$-forms. It is likely that the hyperbolic metric for our examples also enjoys a uniform spectral gap, but we leave that question to future work. See \cref{sec:questions} for more discussion.\\

Our second main result illustrates a tension between the four properties. 
\begin{theorem}\label{thm:superpolytorsion}
    For any Riemannian rational homology 3-sphere $M$ and any $k > 0$, we have
    $$|\Ht M| \geq c \left({h_0(M) h_1(M) h_2(M)}\vol(M)\right)^k$$ for some constant $c$ depending only on $k$.
\end{theorem}

In particular, if $1/h_0$, $1/h_1$, and $1/h_2$ are $o(\vol(M)^\varepsilon)$, then $|\Ht M|$ must grow superpolynomially fast in $\vol(M)$. This gives a negative answer to \cref{question:1}. Note that the Cheeger constants all have units of $length^{-1}$, so both sides of the inequality in \cref{thm:superpolytorsion} are dimensionless. This theorem is reminiscent of the fact that an expander family of Riemannian surfaces must have unbounded genus \cite{yang_eigenvalues_1980}. Applying \cref{thm:cheeger}, we obtain a relationship between the torsion homology of $M$ and its spectral geometry:
\begin{corollary}{\label{cor:spectralrestatement}}
    For any rational homology 3-sphere $M$ with a metric of 1-bounded geometry and any $k>0$, we have $$|H_1(M,\Z)_{tors}|>c\left(\frac 1 {\diam M}  \lambda_1^0\sqrt{\lambda_1^1}\sqrt{\vol(M)}\right)^k$$
    for some constant $c$ depending only on $k$. Here, $\lambda_1^0$ and $\lambda_1^1$ are the spectral gaps of the Hodge Laplacian on functions and coexact 1-forms respectively.
\end{corollary}
This result was motivated by the work of Bergeron--Venkatesh on torsion homology growth in towers of covers. A more detailed comparison with their results is in \cref{subsec:comparebv}. The proof of \cref{thm:superpolytorsion} combines ideas coming from Gromov's topological overlap theorem with bounds on the diameter of the universal abelian cover of $M$.\\

Gromov discovered that high dimensional expansion is related to topological overlap phenomena \cite{gromov_singularities_2010, dotterrer_expansion_2018}. The manifolds we construct give new examples of such phenomena. The \emph{width} of a Riemannian $n$-manifold over $\R^k$ is $\inf_f \sup_{p\in \R^k} |f^{-1}(p)|$, where $f$ ranges over all piecewise smooth maps from the manifold to $\R^k$ and $|f^{-1}(p)|$ is the $n-k$ dimensional volume of the fiber over $p$.

Gromov showed that the width of an $n$-manifold over $\R^n$ is bounded above by a constant independent of $n$ \cite[Section 2.9f]{gromov_singularities_2010}. For example, every 3-manifold can be expressed as degree 3 ramified cover of $S^3$. Guth constructed examples of $n$-manifolds with large width over $\R^k$ for $k<n/2$; in this regime, one can embed a very highly connected $k$-complex into an $n$-manifold. Guth asked whether there exist any examples of $n$-manifolds with large width (say $\gtrsim \vol(M)^{1-\varepsilon}$) over $\R^k$ for some $k\geq n/2$; \cref{thm:3} gives the first such examples:
\begin{theorem}\label{thm:3}
    There is a sequence of Riemannian 3-manifolds $M_i$ with bounded geometry and volume going to infinity such that for any map $f: M_i \to \R^2$, there is a point in $\R^2$ with preimage of total length $\gtrsim \vol(M_i)^{1-\varepsilon}$. In other words, the widths of these manifolds over $\R^2$ are $\gtrsim \vol(M_i)^{1-\varepsilon}$.
\end{theorem}

It remains an intriguing open question to decide whether the manifolds $M_i$ could be taken to all be homeomorphic to a fixed manifold, say $S^3$. See \cite[Naive conjecture 4]{guth_metaphors_2010} for more background on this question and its relationship with other questions in systolic and metric geometry.

\begin{subsection}{Torsion homology growth and spectral expansion}\label{subsec:comparebv}
Motivated by number theoretic considerations, Bergeron and Venkatesh made the following conjecture: 

\begin{conjecture}[\cite{bergeron_asymptotic_2013}]\label{conj:bv}
Suppose that $M_0 \xleftarrow{} M_1 \xleftarrow{} \dots$ is tower of covers of congruence arithmetic hyperbolic 3-manifolds, of injectivity radius approaching $\infty$. Then $|\Ht{M_i}|$ must grow exponentially in $vol(M_i)$, and in fact
$$\lim_{i\to \infty} \frac {\log{|\Ht{M_i}|}}{\vol{M_i}} = \frac 1 {6\pi}.$$
\end{conjecture}

Similar conjectures were made independently by Lê and Lück \cite{le_growth_2018, luck_approximating_2013}. For a geometer, the obvious question is: what are the geometric features of arithmetic hyperbolic 3-manifolds that force large torsion homology? Bergeron--Şengün--Venkatesh proved the following theorem which has only geometric hypotheses:

\begin{theorem}[\cite{bergeron_torsion_2016}]\label{thm:bsv}
Let $M_0 \xleftarrow{} M_1 \xleftarrow{} \dots$ be a tower of closed hyperbolic 3-manifolds with injectivity radius going to $\infty$. Assume the following
conditions are satisfied:
\begin{enumerate}
    \item  ‘Few small eigenvalues’: For every $\varepsilon > 0$ there exists some positive real number $c$
such that
$$\limsup_{i\to \infty} \frac 1 {\vol(M_i)} \sum_{0 < \lambda < c} |\log \lambda| \leq \varepsilon$$
Here $\lambda$ ranges over eigenvalues of the Hodge Laplacian acting on all differential forms on $M_i$.

\item ‘Small Betti numbers’: $$b_1(M_i,\mathbb Q) = o\left(\frac{\vol(M_i)}{\log\vol(M_i)}\right)$$

\item ‘Low cycle complexity’: There is a constant $C$ such that for each $M_i$, there exists a basis of immersed surfaces $S_i$ spanning $H_2(M_i, \R)$ each of genus $< \vol(M_i)^C$.
\end{enumerate}
Then $$\lim_{i\to \infty} \frac {\log{|H_1(M_i, \Z)_{tors}|}}{\vol{M_i}} = \frac 1 {6\pi}.$$

\end{theorem}
This theorem is morally very similar to our \Cref{cor:spectralrestatement}; they both say that a sequence of manifolds which are good expanders in all dimensions must have large torsion homology. Nonetheless, the two theorems are proved in rather different ways. Bergeron--Şengün--Venkatesh require their sequence of manifolds to Benjamini--Schramm converge to $\mathbb H^3$, and deduce the abundance of torsion from the non-vanishing of the $L^2$ analytic torsion of $\mathbb H^3$. In \Cref{thm:superpolytorsion}, we impose no hypothesis on the injectivity radius of our manifolds, nor even require that they be hyperbolic. We do conclude a weaker result, superpolynomial torsion growth instead of exponential. Regardless, this shows that expansion alone is sufficient to force large torsion homology. On the other hand, large injectivity radius alone may not be sufficient to force large torsion homology. Brock and Dunfield constructed a sequence of hyperbolic integer homology spheres which Benjamini--Schramm converges to $\mathbb H^3$ \cite{brock_injectivity_2015}. 

We currently don't know any tower of covers which can be proved to satisfy the ``few small eigenvalues'' condition of \Cref{thm:bsv}. As injectivity radius approaches $\infty$, the spectral gap of the 1-form Laplacian necessarily goes to zero. Therefore, verifying the ``few small eigenvalues'' condition requires understanding not only the spectral gap, but the next smallest eigenvalues as well. On the other hand, it is more tractable to control $h_1(M_i)$, and \cref{thm:0} gives examples of manifolds for which \Cref{thm:superpolytorsion} applies.

\begin{comment}
One might ask whether \cref{thm:superpolytorsion} has something to say for hyperbolic 3-manifolds with $\inj(M)\to \infty$. This is indeed possible, but $h_1(M)$ must lie in a relatively narrow band of values. When $M$ is a hyperbolic rational homology 3-sphere, $\inj(M)\to \infty$ implies that $h_1(M)\to 0$. To see this, let $\gamma$ be the shortest geodesic in $M$. Then $\gamma$ has a radius $\inj(M)$ embedded tubular neighbourhood. Any filling of $\gamma$ has area $|\gamma|\exp(\inj(M))$, so $h_1(M)\leq \exp(-\inj(M))$. On the other hand, the volume of the tube is $|\gamma|\exp(2\inj(M))=2\inj(M)\exp(2\inj(M))$. 
\end{comment}

\end{subsection}

\begin{subsection}{Organization of the paper}
    \cref{sec:cheeger} is an introduction to Cheeger constants. \Cref{sec:comparison} is devoted to proving a comparison theorem between simplicial and de Rham Cheeger constants. This allows us to give a simple proof of a Cheeger-Buser for $i$-forms. We also prove \Cref{cor:spectralrestatement}. \Cref{sec:dehn} is a warmup for \cref{sec:promotecomplexes} where we explain the construction of the examples in \cref{thm:0}. Finally, in \cref{sec:topoverlap} we explain the proofs of \cref{thm:superpolytorsion} and \cref{thm:3}. This section can essentially be read independently of Sections \ref{sec:comparison}-\ref{sec:promotecomplexes}. In \cref{sec:questions}, we conclude with some questions.
\end{subsection}

\section*{Acknowledgements}
I would like to thank Elia Portnoy and Larry Guth for teaching me about expanders, sharing many interesting questions about them, and giving valuable feedback on drafts of this paper. I'd also like to thank Tom Mrowka, Vikram Giri, Amina Abdurrahman, Anshul Adve, and Ben Lowe for many stimulating conversations around this circle of ideas.

\end{section}
\section{Cheeger constants}\label{sec:cheeger}

Let $X$ be either a finite simplicial complex or a closed Riemannian manifold. A cochain of $X$ is either a simplicial cochain in the case of a simplicial complex or a de Rham cochain in the case of a Riemannian manifold. A chain in $X$ is a simplicial chain in the simplicial case or a smooth current in the case of a manifold. Let $C_i(X,R)$ denote the space of $i$-chains in $X$ with coefficients in $R$. We always augment our chain complex including $C_{-1}(X,R) \cong R$ so that $H_0(X,R)=0$ whenever $X$ is connected. Let $C^i(X,R)$ denote the space of $i$-cochains. 

When $X$ is a finite simplicial complex, we use $\langle -,-\rangle$ to denote the inner product on $C_i$ or $C^i$ coming from the preferred basis of $i$-simplices. When $X$ is a Riemannian manifold, $\langle -, - \rangle$ refers to the Riemannian inner product. In a manifold, $\norm{\cdot}_p$ refers to the $L^p$ norm on currents or forms with respect to the Riemannian metric. In a simplicial complex, it refers to the $L^p$ norm in the standard basis where all simplices have norm 1.

Let $\partial: C_i \to C_{i-1}$ be the boundary map and let $d: C_i \to C_{i+1}$ be the adjoint of $\partial$ with respect to $\langle -,- \rangle$. Abusing notation, we also use $d$ to denote the coboundary map $d: C^i \to C^{i+1}$. Given a chain $\alpha$, we say that $\beta$ is a \emph{filling} of $\alpha$ if $\partial \beta=\alpha$. Given a cochain $\alpha$, we say that $\beta$ is a \emph{cofilling} of $\alpha$ if $d \beta=\alpha$. The Cheeger constants of $X$ measure the difficulty of finding fillings or cofillings of small norm.

\begin{definition} For a choice of norm $\norm{\cdot}$ (typically a mass norm or $L^2$) and coefficient ring $R$ (typically $\Z/2$, $\Z$, $\Q$, or $\R$), the $i^{th}$ Cheeger constant of $X$ is defined as

$$h_i(X,\norm{\cdot},R) = \inf_{\substack{\alpha \in C_{i}(X,R)\\ \partial \alpha =0 \\ \alpha \neq 0}} \sup_{\substack{\beta \in C_{i+1}(X,R)\\ \partial \beta = \alpha}} \frac{\norm{\alpha}}{\norm \beta}$$

In other words, $h_i$ is large when every $i$-cycle $\alpha$ has filling $\beta$ of small norm compared to $\alpha$. If $H_i(M,R)\neq 0$, then $h_i(M,\norm{\cdot},R)=0$. When treating spaces with nonvanishing homology, it is useful to define the variant $h_{i,\exact}(X,\norm{\cdot},R)$ which has the same definition as $h_{i}(X,\norm{\cdot},R)$, but the infimum runs instead over all exact $\alpha$.

We also define the corresponding notion for cochains:

$$h^i(X,\norm{\cdot},R) = \inf_{\substack{\alpha \in C^{i}(X,R)\\ d \alpha =0 \\ \alpha \neq 0}} \sup_{\substack{\beta \in C^{i-1}(X,R)\\ d \beta = \alpha}} \frac{\norm{\alpha}}{\norm \beta}$$
Again, we define the analogous quantity $h^i_{\coexact}$, where the infimum runs instead over all coexact $\alpha$.
\end{definition}

The Cheeger constants are sensitive to both the choice of coefficient ring and norm. The \emph{comass norm} on differential $i$-forms is defined as $$\norm{\omega}_\comass = \sup_\xi \omega(\xi)$$ where $\xi$ runs over all simple tangent $i$-vectors of norm 1. The \emph{mass norm} on $i$-currents is defined as $$\norm{x}_{\mass} = \sup \{ \omega(x) \mid \omega \in C^i(X,\R), \norm{\omega}_\comass \leq 1\}.$$ The mass norm generalizes the Riemannian $i$-dimensional volume to currents. The mass norm is comparable to the $L^1$ norm up to a constant depending only on the dimension of $X$. In a simplicial complex, we simply define $\norm{\cdot}_\mass = \norm{\cdot}_1$. Throughout this paper, we will usually work with coefficients in $\R$ and $\norm{\cdot}=\norm{\cdot}_\mass$ or $\norm{\cdot}=\norm{\cdot}_2$.  When left unspecified, we always mean to take the mass norm and $\R$ coefficients; for example, we write $h_i(X)$ instead of $h_i(X,\norm\cdot_\mass,\R)$. 

Some of these Cheeger constants have more familiar interpretations.

\begin{itemize}
\item When $X$ is connected, $h_0(X,\norm{\cdot}_\mass, R)=2/\diam(X)$. Here, it was important that we augmented our chain complex so that every 0-cycle has a filling.
\item $h^1_{\coexact}(X,\norm{\cdot}_\mass,\Z)$ is the usual Cheeger constant.
\item $h_0(X,\norm{\cdot}_2,\R)=h^1_{\coexact}(X,\norm{\cdot}_2,\R)$ is the square root of the spectral gap of the Laplacian acting on functions.
\item When $X$ is an oriented $n$-manifold, Poincar\'e duality implies that $h_i(X,\norm\cdot,R) = h^{n-i}(X,\norm\cdot,R)$
\item $h_1(M,\norm{\cdot}_\mass,\Q)$ is known as the stable isoperimetric ratio for curves
\end{itemize}

A note is warranted about rational coefficients. As we have set things up, $C_i(X,\Q)$ does not make sense when $X$ is a manifold because smooth currents must have real coefficients. We will work instead with $C_{i,\sing}(X,\Q)$, the space of polyhedral singular chains with coefficients in $\Q$. We denote the resulting Cheeger constant $h_{i,\sing}(X,\norm{\cdot}_\mass,\Q)$. Working with polyhedral chains gives the same Cheeger constant:

\begin{prop}\label{prop:realrational}
Let $X$ be a Riemannian manifold. Then
\begin{equation*}
    h_{i,\sing}(X,\norm{\cdot}_\mass,\Q) = h_{i,\sing}(X,\norm{\cdot}_\mass, \R) = h_i(X,\norm{\cdot}_\mass,\R).
\end{equation*}
\end{prop}

To see this, we will need to compare polyhedral chains to smooth currents. In one direction, we can always mollify polyhedral chains to get smooth currents. The reverse direction is provided by the following consequence of Federer and Fleming's deformation theorem:
\begin{theorem}[{\cite[Theorems 4.2.21 and 4.2.24]{federer_geometric_1996}}]\label{thm:fedapprox} Let $T$ be an
$i$-current with $\norm{T}_\mass + \norm{\partial T}_\mass<\infty$ and $\varepsilon > 0$. Then there exists a polyhedral $i$-chain $T'$ and an $i+1$-current $S$ such that $\norm{T - T' - \partial S}_\mass \leq \varepsilon$ and $\norm{S}_\mass < \varepsilon$, and $\norm{T'}_\mass + \norm{\partial T'}_\mass \leq \norm{T}_\mass +\norm{\partial T}_\mass + \varepsilon$. Moreover, if $\partial T$
is polyhedral it is possible to take $\partial T' = \partial T$.
\end{theorem}

\begin{proof}[Proof of \cref{prop:realrational}]
If $H_i(X,\R)\neq 0$, then all three quantities vanish. Assume now that $H_i(X,\R)=0$. The inequality $$h_{i,\sing}(X,\norm{\cdot}_\mass,\Q) \leq h_{i,\sing}(X,\norm{\cdot}_\mass, \R)$$ follows from density of $C_{i,\sing}(X,\Q)$ in $C_{i,\sing}(X,\R)$ in the mass norm. To check the reverse inequality, it suffices to show that fillings of rational $i$-chains by real $i+1$-chains can be approximated by rational fillings of similar norm. Suppose $\alpha \in C_{i,\sing}(X,\Q)$ has a filling $\beta\in C_{i+1,\sing}(X,\R)$.  The finite dimensional subcomplex of $C_{*,\sing}(X,\R)$ spanned by all simplices appearing in $\beta$ has an integer boundary map. Therefore, the equation $\partial \beta'=\alpha$ has a rational solution as close as desired to $\beta$.

Let's argue $h_{i,\sing}(X,\norm{\cdot}_\mass,\R)\geq h_i(X,\norm{\cdot}_\mass,\R)$; the remaining direction is similar and we leave it to the reader. Suppose $\alpha$ is a cycle in $C_{i,\sing}(X,\R)$. We want to construct a filling for $\alpha$ of small norm. Mollify $\alpha$ to get a smooth $\alpha' \in C_i(X,\R)$ with $\norm{\alpha'}_\mass \leq \norm{\alpha}_\mass + \varepsilon$ and $\alpha-\alpha' = d\beta_0$ for some $i+1$-current $\beta_0$ satisfying $\norm{\beta_0}_\mass < \varepsilon$. 
 Next, find some $\beta' \in C_{i+1}(X,\R)$ with $\partial \beta' = \alpha'$. By the definition of the Cheeger constant, we can arrange that $\norm{\beta'}_\mass \leq \frac 1 {h_i(X,\norm{\cdot}_\mass,\R)}\norm{\alpha'}_\mass$. Then $\beta_0 + \beta'$ is an $i+1$-current filling $\alpha$. \cref{thm:fedapprox} gives an approximation of $\beta_0+\beta'$ in the mass norm by a polyhedral chain $\beta''$ which also satisfies $\partial \beta''= \alpha$ and $\norm{\beta''}_\mass \leq \norm{\beta_0 + \beta'}_\mass + \varepsilon$.
 \begin{align*}   
 \norm{\beta''}_\mass &\leq \norm{\beta'}_\mass + \norm{\beta_0}_\mass + \varepsilon\\
 &\leq \frac 1 {h_i(X,\norm{\cdot}_\mass,\R)}\norm{\alpha'}_\mass + 2\varepsilon\\
 &\leq \frac 1 {h_i(X,\norm{\cdot}_\mass,\R)}(\norm{\alpha}_\mass + \varepsilon) + 2\varepsilon\\
 &\leq \frac 1 {h_i(X,\norm{\cdot}_\mass,\R)}\norm{\alpha}_\mass + O(\varepsilon)
 \end{align*}

 Thus, any $\alpha$ has an efficient filling by a polyhedral $i+1$-chain and we have proven the desired inequality.
\end{proof}

We prefer to state our results in terms of $h_i(X,\norm{\cdot}_\mass,\R)$, but will make use of the the definition in terms of rational polyhedral chains in \cref{sec:topoverlap}. When we use $L^2$ norms, many of the different Cheeger constants become the same:

\begin{prop}\label{prop:l2easy1}
    Suppose $H_i(X,\R)=0$. Then $h_i(X,\norm{\cdot}_2,\R)$ is the square root of the smallest eigenvalue of the operator $\partial d$ acting on $i$-cycles. Similarly, when $H_i(X,\R)=0$, $h^{i}(X, \norm{\cdot}_2,\R)$ is the square root of the smallest eigenvalue of the operator $d \partial$ acting on $i$-cocycles.
\end{prop}
\begin{proof}
    When $H_i(X,\R)=0$, $d$ is an isomorphism between $i$-cycles and $i+1$-cocycles. Here, we are using the inner product to identify $i$-chains with $i$-cochains. This makes sense both for manifolds and for simplicial complexes. Our definition for $h_i(X,\norm{\cdot}_2,\R)$ is the variational definition for the smallest nonzero singular value of $d$. The eigenvalues of $\partial d$ are the squares of the singular values of $d$, and the result follows. The second statement follows the dual argument using that $\partial$ is an isomorphism between $i$-cocycles and $i-1$-cycles.
\end{proof}
\begin{corollary}\label{cor:fullspectrum}
    Suppose $H_i(X,\R)=0$. Then the square root of the spectral gap of the Hodge Laplacian $\Delta = \partial d + d\partial$ acting on $i$-chains (or equivalently on $i$-cochains) is $\min(h_i(X,\norm{\cdot}_2,\R),h^i(X,\norm{\cdot}_2,\R))$.
\end{corollary}
\begin{proof}
    The Hodge theorem gives a splitting of the spectrum of $\Delta$ into the spectrum of exact and coexact chains. By \cref{prop:l2easy1}, these spectral gaps are given by $h_i$ and $h^i$ respectively.
\end{proof}

\begin{corollary}\label{prop:l2easy0}
    If $H_i(X,\R)=H_{i+1}(X,\R)=0$, then $h_i(X,\norm{\cdot}_2,\R)=h^{i+1}(X,\norm{\cdot}_2,\R)$.
\end{corollary}
\begin{proof}
    This follows from \cref{prop:l2easy1} and fact that the spectrum of $\partial d$ and the spectrum of $d\partial$ agree.
\end{proof}

So in the case of $L^2$ norms on a rational homology sphere, there are only two different Cheeger constants, $h_0=h^1=h_2=h^3$ and $h_1=h^2$.

The following variant of the Cheeger constant will also be useful:
\begin{definition}

$$\hprime(X,\norm{\cdot},R) = \inf_{\substack{\alpha \in C_{2}(X,R) \\ \alpha \neq 0}} \sup_{\substack{\beta \in C_{1}(X,R)\\\gamma \in C_2(X,R)\\ \partial \gamma = 0\\ \alpha=d \beta + \gamma}} \frac{\norm{\alpha}}{\norm \beta + \norm \gamma}$$

In other words, $\hprime(X)$ is large if every 2-chain $\alpha$ (now not necessarily closed or coclosed) has an efficient decomposition into a coexact piece $d\beta$ and a closed piece $\gamma$, and the coexact piece has an efficient primitive.
\end{definition}

\begin{prop}\label{prop:l2easy}
Suppose $H^2(X,\R)=0$. Then $$\frac 1 {h^2(X,\norm{\cdot}_2,\R)} \leq \frac 1 {\hprime(X,\norm{\cdot}_2,\R)} \leq \frac 1 {h^2(X,\norm{\cdot}_2,\R)} + 1$$.
\end{prop}
\begin{proof}
    The condition on $H^2(X,\R)$ guarantees that $h^2$ is nonzero. The first inequality is true essentially by definition: if $\alpha$ is a coclosed 2-chain and $\alpha = d\beta + \gamma$ is a decomposition of the kind considered in the definition of $\hprime$, then $\gamma$ is both closed and closed, hence 0. To see the second inequality, observe that with respect to the $L^2$ norm, any 2-chain $\alpha$ has an orthogonal decomposition into a closed part and a coexact part. The coexact part has a primitive of size $\frac 1 {h_2(X,\norm{\cdot}_2,\R)}\norm{\alpha}_2$ while the closed part has norm at most $\norm{\alpha}_2$.
\end{proof}

\begin{lemma}[Norm comparison]\label{lem:normcomparison}
Let $X$ be a simplicial complex. For any $i$ and any $p,q \in [0,\infty]$, we have $$h_i(X,\norm\cdot_p,\R) \leq \poly(\vol(X)) h_i(X,\norm\cdot_q,\R)$$
\end{lemma}
\begin{proof}
    $L^p$ norms on a finite simplicial complex differ by at most a factor of $\poly(\vol(X))$.
\end{proof}

The Leray-Serre spectral sequence computes the homology groups of the total space of a fibration in terms of the homology groups of the fiber and the base space. When the homology of the total space is trivial, we can make this computation quantitative and give lower bounds on the Cheeger constants of the total space in terms of the Cheeger constants of the fiber and the base space. We will need only the simplest version of this result:
\begin{lemma}\label{lem:lerayserre}
    Suppose $E$ and $B$ are finite connected graphs, each of degree $<D$. Suppose $\pi: E \to B$ is a surjective simplicial map, and for any vertex $v\in B$, $\pi^{-1}(v)$ is connected and $|\pi^{-1}(v)| < C$. Then the Cheeger constants for $E$ and $B$ are comparable. Specifically, for any $1\leq p \leq \infty$, we have

    $$\frac 1 {h_0(E,\norm{\cdot}_p)} \leq 2D\left(\frac 1 {h_0(B,\norm{\cdot}_p)} +1 \right)\max_{v\in B} \left(\frac 1 {h_0(\pi^{-1}(v),\norm{\cdot}_p)} + 1 \right)$$

    and
    $$\frac 1 {h^1_{\coexact}(E,\norm{\cdot}_p)} \leq C\left(\frac 1 {h^1_{\coexact}(B,\norm{\cdot}_p)} +1 \right)\max_{v\in B} \left(\frac 1 {h^1_{\coexact}(\pi^{-1}(v),\norm{\cdot}_p)} + 1 \right)$$

\end{lemma}
\begin{proof}
    Let $\alpha$ be a 0-cycle in $E$. (Recall that 0-cycle means a chain that it is closed in the augmented chain complex, eg a pair of oppositely oriented points.) We want to construct an efficient filling for $\alpha$. Let $\eta$ be the least norm filling for $\pi(\alpha)$. Let $\pi^{-1}(\eta)$ be any choice of lift of $\eta$. With this choice, $\partial \pi^{-1}(\eta)-\alpha$ restricts to a 0-cycle on $\pi^{-1}(v)$ for any $v\in B$. For each $v$, let $\beta_v$ be a filling for $\partial \pi^{-1}(\eta) - \alpha$. Then $\sum_{v\in B} \beta_v + \pi^{-1}(\eta)$ is a filling for $\alpha$. Let us estimate the norm of this filling.

    \begin{align*}
        \norm{\pi^{-1}(\eta)}_p &\leq \norm{\eta}_p\\
        &\leq \frac 1 {h_0(B,\norm{\cdot}_p)} \norm{\alpha}_p\\
        \norm{\sum_{v\in B} \beta_v}_p &\leq \max_{v\in B} \frac 1 {h_0(\pi^{-1}(v), \norm{\cdot}_p)} \norm{\partial \pi^{-1}(\eta) - \alpha}_p\\
        &\leq \max_{v\in B} \frac 1 {h_0(\pi^{-1}(v), \norm{\cdot}_p)} \left(\frac {2D} {h_0(B,\norm{\cdot}_p)} + 1\right)\norm{\alpha}_p\\
    \end{align*}

    In the last line, we used the fact that $\norm{\partial e}_p \leq 2D \norm {e}_p$. Combining these two inequalities, we find that our filling has small enough norm to validate the first inequality. Now we do the dual contruction to get the second inequality. Let $\alpha$ be a coexact 1-cocycle in $E$. For each vertex $v\in B$, let $\beta_v$ be a $0$-cochain which is a cofilling for the restriction of $\alpha$ to $\pi^{-1}(v)$. Then $\alpha - \sum_{v\in B} d\beta_v$ is a coexact 1-cocycle which vanishes on any edge of $E$ whose endpoints project to the same vertex in $B$. Since the fibers of $\pi$ are connected, any coexact 1-cocycle with this property must take the same value on the different lifts of an edge of $B$ to $E$. Therefore, $\alpha - \sum_{v\in B}=\pi^{-1}(\gamma)$ for some coexact 1-cocycle $\gamma$ on $B$. Let $\eta$ be the smallest cofilling for $\gamma$. Then $\sum_{v\in B} \beta_v + \pi^{-1}(\eta)$ is a cofilling of $\alpha$. Let's estimate the norm of this cofilling.

    \begin{align*}
        \norm{\sum_{v\in B} \beta_v}_p &\leq \max_{v\in B} \frac 1 {h^1_{\coexact}(\pi^{-1}(v), \norm{\cdot}_p)} \norm{\alpha}_p\\
        \norm{\pi^{-1}(\eta)}_p & \leq  C\norm{\eta}_p\\
        &\leq C\frac 1 {h^1_{\coexact}(B,\norm{\cdot}_p)} \norm{\gamma}_p\\
        &\leq C\frac 1 {h^1_{\coexact}(B,\norm{\cdot}_p)} \left(\norm{\alpha}_p + \norm{\sum_{v\in B} \beta_v}_p \right)\\
        &\leq C\frac 1 {h^1_{\coexact}(B,\norm{\cdot}_p)} \left(\max_{v\in B} \frac 1 {h^1_{\coexact}(\pi^{-1}(v)), \norm{\cdot}_p)}  +1 \right) \norm{\alpha}_p\\
    \end{align*}
    Combining these two inequalities gives the desired bound on $h^1_{\coexact}(E,\norm{\cdot}_p)$.
\end{proof}

\begin{definition}
    We say that a Riemannian manifold has \emph{$c$-bounded geometry} if its sectional curvatures are less than $c$ in absolute value and its injectivity radius is greater than $\frac 1 c$. We say that a simplicial complex has $c$-bounded geometry if the maximum number of simplices containing any given vertex is at most $c$.
\end{definition}

We will often build a filling as a composition of several immersed cobordisms. This motivates the following definition:
\begin{definition}
    Given an $i$-cycle $c$, we say that $c$ can be \emph{transported} to an $i$-cycle $c'$ for \emph{cost} $x$ if there is an $i+1$-cycle $C$ with $\partial C=c'-c$ and $\norm{C}\leq x\norm c$. We say that $c$ can be filled for cost $x$ if $c$ can be transported to 0 for cost $\leq x$. If $c$ can be transported to $c'$ for cost $x$, and $c'$ can be transported to $c''$ for cost $y$, then $c$ can be transported to $c''$ for cost $xy + x$.
\end{definition}

\begin{lemma}\label{lem:expfill}
Let $X$ be a finite simplicial complex. Suppose that there exist constants $a$ and $x$ satisfying $0 < a < 1$ and such that any $i$-cycle $c$ can be transported to an $i$-cycle $c'$ satisfying $\norm{c'} \leq a\norm{c}$ for cost $x$. Then $h_{i}(X,\norm \cdot) \geq \frac{1-a}x$. In particular, $H_i(X)=0$.
\end{lemma}
\begin{proof}
    Since each $i$-cycle can be transported to a smaller cycle, we can transport any cycle to zero using an infinite series of $i$-chains. This sum converges since its norm is bounded above by a geometric series of ratio $a$. This is the efficient filling we need.
\end{proof}

\section{Comparing simplicial and de Rham Cheeger constants}\label{sec:comparison}
Let $M$ be a Riemannian $n$-manifold with 1-bounded geometry. Let $X$ be a smooth triangulation of $M$ such that every simplex has uniformly controlled geometry (ie, every simplex is diffeomorphic to a regular simplex with unit sidelengths by a map with all derivatives bounded above by some constant $K$), and every simplex is contained in a ball of radius $\frac 1 {100n}$. Such a triangulation always exists for some $K$ depending only on $n$ (see Theorem 3 of \cite{boissonnat_delaunay_2015}). It is usually easier to work with simplicial cochains than de Rham cochains. One reason is that $d$ is a bounded operator on simplicial cochains. 
\begin{prop}\label{prop:dbounded}
    Let $\norm{\cdot}_p$ be an $L^p$ norm for some choice $1\leq p \leq \infty$. Let $c$ be a cochain in $C^*(X)$. Then $\norm{dc}_p \leq A\norm{c}_p$ for some $A$ depending only on the local degree of $X$. Similarly, if $c$ is an $i$-chain in $C_*(X)$, then $\norm {\partial c} \leq A \norm c_p$ for some $A$ depending only on $i$.
\end{prop}
Thus, simplicial cochains can be thought of as de Rham cochains without any high frequency components. The Cheeger constants for $M$ do not depend on these high frequency components, and can be recovered up to a constant factor from the Cheeger constants of $X$:

\begin{prop}\label{prop:simplicial}
    Let $\norm{\cdot}_p$ be an $L^p$ norm for some choice $1\leq p \leq \infty$. Then the Cheeger constants for $M$ and $X$ with respect to $\norm{\cdot}_p$ are comparable. That is, for any $i$ we have 
    $$\frac 1 {h^i(M,\norm{\cdot}_p)} \leq \frac c {h^i(X,\norm{\cdot}_p)} + d$$
        $$\frac 1 {h^i(X,\norm{\cdot}_p)} \leq \frac c {h^i(M,\norm{\cdot}_p)} + d$$ 
        for some positive constants $c,d$ depending only on $n$, $p$, and $K$. The same holds when $h_i$ and $h^i$ are replaced with $h_{i,\exact}$ and $h^i_{\coexact}$.
\end{prop}
\begin{proof}
    Recall that $C_*(X)$ denotes the simplicial chain complex of $X$ and $C^*(M)$ denotes the be the de Rham complex of $M$. It suffices to find chain maps $f: C_*(X) \to C^{n-*}(M)$ and $g:C^{*}(M) \to C_{n-*}(X)$ which are bounded in $L^p$, such that $g \circ f$ and $f\circ g$ are chain homotopic to the identity via homotopy operators $I$ and $J$ of $L^p$ norm bounded by $c$. Suppose we are in possession of such a pair of maps $f$ and $g$. Let's try to show that any $i$-cycle $\alpha\in C_i(X)$ has an efficient filling. Instead of trying to fill $\alpha$, first find a cofilling $\beta \in C^{n-i-1}(M)$ for $f\alpha \in C^{n-i}(M)$. Then $g\beta + I \alpha$ is a filling of $\alpha$, because
    \begin{align*}
    d(g\beta + I \alpha) &= d g \beta + d I \alpha\\
    &= g d \beta - I d \alpha + \alpha - g\circ f \, \alpha\\
    &= g\circ f \, \alpha + \alpha - g\circ f\, \alpha\\
    &= \alpha
    \end{align*}

    The norm of this filling is controlled by the $i^{th}$ Cheeger constant of $M$ and the norms of $f$,$g$, and $I$: $$\norm{g\beta + I \alpha}_p \leq \norm{g}_p \frac 1 {h^{n-i}(M, \norm{\cdot}_p)} \norm{\alpha}_p + \norm{I}_p \norm{\alpha}_p$$ By Poincar\'e duality, $h^{n-i}=h_i$, so the first inequality is proven. The second inequality is obtained by interchanging the roles of $M$ and $X$. The same arguments go through for $h_{i,\exact}$ and $h^i_{\coexact}$ after restricting $\alpha$ to be exact or coexact respectively.

    Now let us construct $f$ and $g$. First we sketch the intuition behind the construction. Suppose there are $N$ vertices in the triangulation. We form an $nN$ dimensional family of embeddings of $X$ in $M$ by wiggling the $N$ vertices independently. Any wiggling of the vertices can be extended to a wiggling of the higher simplices of $X$. Call this family of embeddings $H$. Any $h_0 \in H$ gives rise to an induced map $h_0: C^*(M) \to C^*(X)$ which integrates an $i$-form over an $i$-simplex. $h_0$ itself may not be bounded in $L^p$. But if we average over all $h_0 \in H$, we get a map $h$ that is bounded in $L^p$. By Poincar\'e duality, $h$ can also be regarded as a map from $C^*(M) \to C_{n-*}(X)$.
    
    Let's make the construction more precise. For any $i$-cell $c$ of $X$, let $N_\varepsilon(c)$ denote the radius $\varepsilon$ neighbourhood of $c$ in $M$. Fix $\varepsilon = \frac 1 {10K\, n!}$; for an $\varepsilon$ this small, $N_\varepsilon(c)$ intersects $N_\varepsilon(c')$ if and only if $c$ and $c'$ are incident. Choose for each vertex $v$ of $X$, a smooth $n$-form $f(v)$ supported in $N_\varepsilon(v)$, so that $\int f(v)=1$ and all the derivatives of $f(v)$ are bounded by constants depending only on $n$ and $K$. This $n$-form may be interpreted as a probability distribution over all possible wigglings of the vertex $v$. Then extend $f$ to higher skeleta of $v$ inductively so that the following conditions are satisfied:
    \begin{itemize}
    \item $f(\partial c)=df(c)$
    \item For each $i$-cell of $X$, $f(c)$ is supported in $N_\varepsilon(c)$.
    \item The norms of $f(c)$ and its derivatives are bounded by constants depending only on $n$ and $K$.
    \end{itemize}
   When we want to extend $f$ to an $i$-cell $c$ having already defined $f(\partial c)$, we define $f(c)$ as a primitive for $f(\partial c)$ obtained using the Poincar\'{e} lemma. $N_\varepsilon(\partial c)$ can be mapped to the neighbourhood of a standard regular $i$-simplex in $\R^n$ by a map with all derivatives bounded by $K$. Therefore, the Poincar\'e lemma gives a primitive with the desired control on the derivatives and norms of $f(c)$. One can inductively show that for any cell $c$ and its dual cell $c^*$, we have $\int_{c^*} f(c)=1$. In particular, $f$ maps the fundamental class of $X$ to the constant function 1 on $M$.

    Let $X^*$  be the dual complex of $X$. Applying the construction above to $X^*$, we get a map $g^*:C_*(X^*) \to C^{n-*}(M)$. Poincar\'{e} duality gives a map $g:C^*(M) \to C_{n-*}(X)$. One difference in the construction of $g$ is that the cells of $X^*$ are not simplices but polyhedra. The easiest way around this is to view cells of $X^*$ as unions of simplices in $X^b$, the barycentric subdivision of $X$. Since each simplex of $X$ is $K$-bilipschitz to a regular $i$-simplex, there is an upper bound on the local degree of $X$. Therefore, each cell of $X^*$ is a union of a bounded number of simplices in $X^b$ and the inclusion map $C_*(X^*) \to C_*(X^b)$ is bounded in norm. So the right way to define $g^*$ is as a composition of maps $C_*(X^*) \to C_*(X^b) \to C^{n-*}(M)$.

    Let $pt_X$ be the class of a point in $H_0(X)$, and let $pt_M$ be the class in $H^n(M)$ which integrates to 1 on the fundamental class. By construction, $f(pt_X)=pt_M$. Since $g^*$ maps the fundamental class on $C_n(X^*)$ to the constant function 1 in $C^0(M)$, we have that $g(pt_M)=pt_X$.

    In what follows, we represent linear maps $C^i \to C^i$ using integral transforms. Let $\pi_1$ and $\pi_2$ be projection maps from $M\times M$ onto each of its factors. Let $\Omega^{i,n-i}(M\times M)$ be the space of sections of $\pi_1^*(\wedge^iT^*M) \otimes \pi_2^*(\wedge^{n-i} T^*M)$, ie differential forms on $M\times M$ with degree $i$ in the first factor and degree $n-i$ in the second factor. A kernel $h_i \in \Omega^{i,n-i}(M\times M)$ gives rise to a map $C^i \to C^i$ by the rule $\alpha \mapsto \int_y h_i \wedge \pi_2^*(\alpha)$. Here, $\int_y$ means fiberwise integration for the bundle $M\times M \xrightarrow{\pi_1} M$.
        
    Let $k_i$ be the integral kernel for the operator $f\circ g$ in degree $i$. It has the explicit formula \begin{equation}\label{eqn:kdef}k_i=\sum_{\text{$c$ an $i$-cell of $X$}} f(c) \otimes g^*(c^*)\end{equation}
    where $c^*$ is the cell of $X^*$ dual to $c$. 
    
    We initially assumed that each simplex of $X$ or $X^b$ is contained in a ball of radius $\frac 1 {100n}$, so $k(-,y)$ is surely supported on a ball of radius $\frac 1 {10}$ around $y$. Moreover, each point $y\in M$ is in the image of $N_\varepsilon(c)$ for a bounded number of cells $c\in X$ or $c\in X^b$. So the number of nonvanishing terms in the formula for $k_i(-,y)$ from \cref{eqn:kdef} is bounded by a function of $K$ and $n$. Our control on the norms and derivatives $f(c)$ and $g(c)$ now gives similar control on $k_i(-,y)$.

    \newcommand{\A}{\Delta_{1/10}}
    Let $\A\subset M\times M$ be the radius $\frac 1 {10}$ neighbourhood of the diagonal.  Define $v:\A \times [1,\infty) \to M$, by $v(x,y,t) = (y+t\overrightarrow {yx},y)$. The meaning of $y+t\overrightarrow{yx}$ is that one takes the unique shortest geodesic from $y$ to $x$, and then extends it by a factor of $t$ to arrive at $y+t \overrightarrow{yx}$. We will sometimes use the shorthand $v_T(x,y)=v(x,y,T)$. Also define the projection maps $\pi_1: \A\times [0,1]\to M$, $\pi_1(x,y,t)=x$, $\pi_2(x,y,t)=y$.

    Fix a parameter $\varepsilon$, and let $\delta^\varepsilon_n\in \Omega^{n,0}(M\times M)$ be the integral kernel for a smooth approximation to the identity which is supported in a $\varepsilon$-neighbourhood of the diagonal. Define $\delta^\varepsilon_i \in \Omega^{i,n-i}(M\times M)$ for $i=0,\dots,n-1$ by extending $\delta^\varepsilon_n$ to a chain map.

Now we define $j_i^\varepsilon $, a candidate for the integral kernel for $J$:

$$j_i^\varepsilon = \int_{t} v^*(k_i  - \delta^\varepsilon_i)$$
Here $\int_t$ is a fiberwise integral for the bundle $M\times M \times [1,\infty) \to M\times M$. We've essentially defined $j_i^\varepsilon$ to be the integral kernel for the straight line homotopy between $f\circ g$ and an $\varepsilon$-approximation to the identity. Let's check this; we need to prove that for all $i$,
\begin{equation} \int_y dj_i^\varepsilon \wedge \pi_2^*\alpha + j_{i+1}^\varepsilon\wedge \pi_2^*d\alpha + (k_i-\delta_i^\varepsilon) \wedge \pi_2^*\alpha=0.\label{eqn:homotopy}\end{equation}
Before proceeding with the proof of \cref{eqn:homotopy}, we record two observations.

\textbf{Observation 1:} $$\int_t v^* (j_{i+1}^\varepsilon)=0$$ Indeed,
\begin{align*}
\int_t v^*(j_{i+1}^\varepsilon) &= \int_t v^*(\int_{t'} v^*(k_{i+1}-\delta_{i+1}^\varepsilon)\\
&=\int_t \int_{t'} v^*(v^*(k_{i+1}-\delta^\varepsilon_{i+1}))
\end{align*}
The final term is a fiberwise integral for the trivial bundle $M\times M \times [1,\infty) \times [1,\infty) \to M\times M$. This term vanishes because $v\circ v : M\times M \times [1,\infty) \times [1,\infty)\to M\times M$ does not have full rank on the fibers $x \times y \times [1,\infty) \times [1,\infty)$.

\textbf{Observation 2:} Suppose $a\in \Omega^{i,n-i}(M\times M)$ is a smooth integral kernel supported in $\A$ and taking values in closed forms. For the purposes of this observation, when $i=n$, we consider an $n$-form to be closed if it integrates to zero, ie it is closed in the augmented de Rham complex. Then for any $i$-form $\alpha$, $$\lim_{T\to \infty} \int_y v_T^*(a)\wedge \pi_2^* \alpha = 0$$ and the limit converges uniformly over $M$. Thinking of $a$ like a matrix, each column is a closed $i$-form, which is Poincar\'{e} dual to an $(n-i)$-cycle. The effect of $v_T^*$ is to shrink this $(n-i)$-cycle down by a factor of $T$. Let's write this more formally. First, note that $v_T^*(a \wedge \alpha)$ is supported in a $\frac 1 T$ neighbourhood of the diagonal. Choose a point of interest $x_0\in M$. Let $\phi: \R^n \to M$ be a local coordinate patch with $\phi(0)=x_0$. Let $\phi_T(x') = \phi(\frac 1 T x')$ and define $\Phi_T: \R^n\times \R^n \to M\times M$  by $(x',y')\mapsto (\phi_T(x'),\phi_T(y'))$. $\Phi_T$ is zooming in near $x_0$ at a length scale of $\frac 1 T$. Near $x_0$ we have

\begin{align*}
\int_y v_T^* (a)\wedge \pi_2^* \alpha &= \int_{y\in B(\frac 1 T,x)} v_T^*(a \wedge \pi_2^* \alpha)\\
&= (\phi_{T}^{-1})^*\int_{y'} \Phi_T^*(v_T^*(a) \wedge \pi_2^* \alpha)\\
&= \frac 1 {T^i} (\phi_T^{-1})^* \int_{y'} \Phi_T^*(v_T^*(a)) \wedge T^i\Phi_T^*(\pi_2^*\alpha)
\end{align*}

As $T\to \infty$, $\Phi_T^* v_T^*(a)$ converges uniformly to a translation invariant integral kernel in $\Omega^{i,n-i}(\R^n \times \R^n)$ supported in a neighbourhood of the diagonal, and $T^i\Phi_T^*\pi_2^*\alpha$ converges uniformly to a translation invariant $i$-form. Therefore, $\lim_{T\to \infty} \int_{y'}\Phi^*(v^*_T(a)) \wedge T^i \Phi^*(\pi_2^*\alpha)$ is the average of all translates of some closed $i$-form on $\R^n$ of compact support. Such an average always vanishes. The derivative of $\frac 1 {T^i} (\phi_T^{-1})^*$ is $O(1)$, so we conclude that $$\lim_{T\to \infty}\int_y v_T^*(a) \wedge \pi_2^*\alpha \bigg|_{x=x_0}=0.$$ This completes the proof of our second observation.

Now let us return to the proof of \cref{eqn:homotopy}. Assume that we have already proven \cref{eqn:homotopy} for larger values of $i$. It will be notationally convenient to define $j^\varepsilon_{n+1}=0$.

For $i<n$ we have

\begin{align}\label{eqn:closed}
&d\int_y v_T^*(j_{i+1}^\varepsilon\wedge \pi_2^*d\alpha + (k_i-\delta^\varepsilon_i) \wedge \pi_2^*\alpha) \\&\qquad= \int_y dv_T^*(j_{i+1}^\varepsilon \wedge d\alpha + d(k_i-\delta_i) \wedge \pi_2^* \alpha )\nonumber\\
&\qquad= \int_y v_T^*( -j_{i+2} \wedge d^2 \pi_2^* \alpha - (k_{i+1}- \delta_{i+1}^\varepsilon) \wedge \pi_2^* d\alpha  + d(k_i -\delta_i^\varepsilon )\wedge \pi_2^* \alpha)\nonumber\\
&\qquad = 0\nonumber
\end{align}

In the third line we used the induction hypothesis and in the last step we used that $k_i$ and $\delta_i^\varepsilon$ are both chain maps.

For any $i$-form $\alpha$, we have
\begin{align*}
    d\int_y j_i^\varepsilon \wedge &\pi_2^*\alpha\\
    &= d\int_y \int_t v^*(k_i - \delta_i^\varepsilon)\wedge \pi_2^* \alpha \\
    &= d\int_y \int_t v^*( j^\varepsilon_{i+1}) \wedge \pi_2^* d\alpha + v^* (k_i - \delta_i^\varepsilon)\wedge \pi_2^* \alpha \\
    &= \int_yj_{i+1}^\varepsilon \wedge \pi_2^*d\alpha + (k_i  - \delta^\varepsilon_i) \wedge \pi_2^*\alpha + \lim_{T\to \infty} \int_y  v_T^*(j_{i+1}^\varepsilon \wedge \pi_2^*d\alpha + (k_i - \delta^\varepsilon_i) \wedge \pi_2^*\alpha)\\
    &\qquad+ \int_y\int_{t} v^*(dj_{i+1}^\varepsilon \wedge \pi_2^*d\alpha + d(k_i  - \delta^\varepsilon_i) \wedge \pi_2^*\alpha)\\
    &= \int_y j_{i+1}^\varepsilon \wedge \pi_2^*d\alpha + (k_i - \delta^\varepsilon_i) \wedge \pi_2^*\alpha + 0\\
    &\qquad + \int_y\int_{t} v^* (-j_{i+2}^\varepsilon \wedge d^2 \pi_2^*\alpha - (k_{i+1} - \delta^\varepsilon_{i+1}) \wedge \pi_2^*d\alpha + d(k_i -\delta^\varepsilon_i) \wedge \pi_2^*\alpha)\\
    &= \int_y j_{i+1}^\varepsilon \wedge \pi_2^*d\alpha + (k_i - \delta^\varepsilon_i) \wedge \pi_2^*\alpha
\end{align*}

In the second step we are using observation 1. In the third step we are using the fiber integration formula. In the fourth step, the term in the limit $T\to \infty$ vanishes thanks to observation 2. Observation 2 applies for $i<n$ by \cref{eqn:closed}, and it applies for $i=n$ by the fact that $k_n - \delta_n^\varepsilon$ is an integral kernel for an operator which kills the class $pt_M$. We also used the induction hypothesis here to replace the term $j_{i+1}^\varepsilon \wedge \pi_2^* d\alpha$. In the last step, we used the chain map equation for $k_i$ and $\delta_i^\varepsilon$.

\begin{remark}
There are two limits in play, one for taking the width of the delta function $\varepsilon \to 0$ and one for taking the integral over $t\in [0,T)$ as $T\to \infty$. There are two options. If you cut off at $T \gg 1/\varepsilon$, the boundary term $\int_y v^*_T(...)$ is negligible. But $\delta^\varepsilon_i \wedge \pi_2^*\alpha$ will contribute to $dj_i^\varepsilon \wedge \pi_2^*\alpha$. On the other hand, if you cut off at $T < 1/\varepsilon$, then you see a nonzero contribution at $t=T$. But $\delta^\varepsilon_i \wedge \pi_2^*\alpha$ doesn't contribute to $dj_i^\varepsilon \wedge \pi_2^*\alpha$. In this setup, we are taking the first option.
\end{remark}

Now let us estimate the norm of the integral transform associated with $j_i^{\varepsilon}$. We will independently bound the norms of its two components, $\int_t v^*(k_i)$ and $\int_t v^*(\delta_i^\varepsilon)$. First, we need to compute the norm of the derivative of $v$; note that $v$ magnifies the norm of any tangent $i$-vector in $\wedge^i T(M\times M \times [1,\infty))$ containing the $t$ direction by a factor of at most $t^{i-1}|x-y|$. (We will be integrating in the $t$ direction, so it is only norms on such $i$-vectors that matter.) Since $k_i$ is supported on a $\frac 1 {10}$ neighbourhood of the diagonal,  $v^*(k_i)$ is supported on the region $t < 1/|x-y|$. Similarly $\delta_i^\varepsilon$ is supported on an $\varepsilon$-neighbourhood of the diagonal, so $v^*(\delta_i^\varepsilon)$ is supported on the region $t < \varepsilon / |x-y|$ and $|x-y|<\varepsilon$. We can now estimate

\begin{align}\label{eqn:kbound}
|v^*(k_i)(x,y,t)|& \leq t^{i-1}|x-y| \sup_{M\times M} |k_i|\nonumber\\
\left|\int_t v^*(k_i)(x,y)\right| &\leq \int_{t=1}^{1/|x-y|}t^{i-1}|x-y|\sup_{M\times M}|k_i|\\
&\lesssim \frac 1{|x-y|^{i-1}}\nonumber
\end{align}

\Cref{eqn:kbound} says that the kernel $\int_{t} v^*(k_i)$ is not too singular since $i\leq n$. It is also supported in a radius 1 neighbourhood of the diagonal. Therefore, there is a uniform $L^1$ bound for $\left(\int_t v^*(k_i)\right)(x_0,-)$ and $\left(\int_t v^*(k_i)\right)(-,y_0)$ for any choice of $x_0,y_0\in M$.

Similarly,
\begin{align}\label{eqn:dbound}
|v^*(\delta_i^\varepsilon)(x,y,t)|& \leq t^{i-1} |x-y| \sup_{M\times M} |\delta_i^\varepsilon|\nonumber\\
\left|\int_t v^*(\delta^\varepsilon_i)(x,y)\right| &\leq \int_{t=1}^{\varepsilon/|x-y|}t^{i-1} |x-y|\varepsilon^{-n}\nonumber\\
&\leq \int_{t=1}^{\varepsilon/|x-y|}t^{i-2} \varepsilon^{-n+1}\nonumber\\
&\lesssim\frac{\varepsilon^{i-n}}{|x-y|^{i-1}}
\end{align}

so

\begin{align*}
  \left|\left(\int_t v^*(\delta_i^{\varepsilon})\right) (x_0,-)\right|_{L^1}&\lesssim \int_{y\in B(\varepsilon, x_0)} \frac{\varepsilon^{i-n}}{|x_0-y|^{i-1}}\\
  &\lesssim \varepsilon
\end{align*}
The same bound holds for $\left|\left(\int_t v^*(\delta_i^{\varepsilon})\right) (-,y_0)\right|_{L^1}$. Combining our bounds on $\int_tv^* k_i$ and $\int_t v^*\delta_i^\varepsilon$, we find that as $\varepsilon \to 0$, $j_i^\varepsilon(x_0,-)$ and $j_i^\varepsilon(-,y_0)$ both converge in $L^1$, uniformly over $x_0,y_0\in M$. Call the limiting integral kernel $j^0_i$. Now let $J_i$ be the integral transform associated with kernel $j^0_i$. Convergence of the kernels in $L^1$ implies that \cref{eqn:homotopy} continues to hold in the limit, ie $J$ is a homotopy between $f\circ g$ and the identity. Moreover, by Young's inequality, the upper bound on the $L^1$ norm of the kernel implies an upper bound on the $L^p$ norm of $J$. The upper bound depends only on $n$, $p$, and $K$. This completes our construction of $J$.

    Constructing the the chain homotopy $I$ is analogous, but much easier since it takes place in a PL setting. The operator $g\circ f: C_*(X) \to C_*(X)$ has the property that $g\circ f(c)$ is supported on simplices at distance at most $2$ from $c$. Recall that we assumed that our simplices had diameter at most $\frac 1 {100n}$. Since the injectivity radius of $M$ is at least 1, the radius $i/n$ balls around any point are embedded in $M$. For each $i$-cell $c$, let $N_0(c) \subset \dots \subset N_n(c)$ be simplicial approximations to the radius $i/n$ balls around $c$. Specifically, we need a nested family of balls such that $N_i(c)$ contains all the simplices at simplicial distance $< 10i$ from $c$ and $N_{i+1}(c)$ contains all simplices at simplicial distance 5 from $N_i(c)$. There is a lower bound on the Riemannian volume of any simplex, so the number of possibilities for the triangulation of $N_i(c)$ is bounded above in terms of $n$ and $K$. It follows that there is an upper bound on $h_j(N_i(c),\norm{\cdot}_p,\mathbb R)$ which is uniform in $i$, $j$, and $c$, and depends only on $p$, $n$ and $K$.
    
    Inductively define $I_i c$ so that
    \begin{itemize}
        \item $I_ic$ is a primitive for $ (1 - g\circ f - I_{i-1} \partial)c$
        \item $\norm{I_ic}_p \lesssim_{p,n,K} 1$.
        \item $I_i c$ is supported on $N_i(c)$.
    \end{itemize}

    Let's first do the base case $i=0$. For any 0-cell $c$, $g\circ f (c)$ is homologous to $c$ and supported on $N_0(c)$. Since $h_0(N_0(c),\norm{\cdot}_p,\mathbb R)$ is bounded above, we can find an efficient primitive for $c-f\circ g(c)$ supported on $N_0(c)$.

    Now suppose we want to extend the definition of $I$ to $i+1$-cells having already defined it for $i$-cells. The induction hypothesis guarantees that $(1-g\circ f -I_{i-1} \partial) c$ is an $i$-cycle whose $p$-norm is $\lesssim_{p,n,K} 1$. Moreover, it is supported on $\bigcup_{c'\in \partial c} N_i(c) \subset N_{i+1}(c)$. Therefore, we can again use the fact that $h_{i+1}(N_{i+1}(c), \norm{\cdot}_p,\R) \lesssim_{p,n,K} 1$ to find an efficient primitive which we declare to be $I_{i+1}c$. This completes our construction of $I$.
    \end{proof}

    Now we can we can prove a Cheeger type theorem relating the spectrum of the Hodge Laplacian to the Cheeger constants.
    \begin{theorem}\label{thm:cheeger}
Let $M$ be a Riemannian $n$-manifold with 1-bounded geometry (ie sectional curvatures less than one in absolute value and injectivity radius $>1$). Let $\lambda^i_1$ be the spectral gap of the Laplacian acting on coexact $i$-forms. Then

$$\frac 1 {C\sqrt{\vol(M)}} \left(\frac 1 {h_{i,\exact}(M,\norm{\cdot}_\mass)} -D\right) \leq \frac 1 {\sqrt{\lambda_1^i}} \leq C \sqrt{\vol(M)}\left(\frac 1 {h_{i,\exact}(M,\norm{\cdot}_\mass)}+D\right)$$

for constants $C$ and $D$ depending only on $n$.
\end{theorem}
\begin{proof}
Let $X$ be the triangulation of $M$ with bounded geometry that we used above. We will use \cref{prop:simplicial} to relate $\lambda^i_1$ to the $L^2$ Cheeger constants of $X$ which in turn are bounded by $L^1$ Cheeger constants via Cauchy-Schwarz.

\begin{align*}
    \frac 1 {\sqrt{\lambda^i_1}} &= \frac 1 {h_{i,\exact}(M, \norm{\cdot}_2,\R)}\\
    &\leq c\frac 1 {h_{i,\exact}(X, \ltwo, \R)} + d & \text{by \cref{prop:simplicial}}\\
    &\leq c \sqrt{\vol X} \frac 1 {h_{i,\exact}(X, \lone, \R)} + d &\text{by Cauchy-Schwarz}\\
    &\leq c\sqrt {\vol X} \left(c' \frac 1 {h_{i,\exact}(M,\lone,\R)} + d' \right) + d &\text{by \cref{prop:simplicial}}\\
    &\leq C \sqrt{ \vol{M} }\left( \frac 1 {h_{i,\exact}(M,\norm{\cdot}_\mass,\R)} + D\right) &\text{for some large $C$, $D$}\\
\end{align*}

Similarly,
\begin{align*}
    \frac 1 {\sqrt{\lambda^i_1}} &= \frac 1 {h_{i,\exact}(M, \norm{\cdot}_2,\R)}\\
    &\geq \frac 1 c\frac 1 {h_{i,\exact}(X, \ltwo, \R)} - d  & \text{by \cref{prop:simplicial}}\\
    &\geq \frac{1}{c\sqrt {\vol X}} \frac 1 {h_{i,\exact}(X, \lone, \R)} - d  &\text{by Cauchy-Schwarz}\\
    &\geq \frac 1 {c\sqrt {\vol X}} \left(\frac 1 {c'} \frac 1 {h_{i,\exact}(M,\lone,\R)} - d' \right) - d &\text{by \cref{prop:simplicial}}\\
    &\geq \frac 1 {C \sqrt{ \vol M }}\left( \frac 1 {h_{i,\exact}(M,\norm{\cdot}_\mass,\R)} - D\right)&\text{for some large $C$, $D$}
\end{align*}
\end{proof}

\begin{remark}
    This kind of theorem is not new, see \cite{lipnowski_geometry_2018, boulanger_cheeger-like_2022, rudd_stable_2021} in the case of coexact 1-forms. The volume dependence is improved in our version, there is no dependence on the diameter of the manifold, and it works for all $i$-forms, so we hope this version will be useful in practice.
\end{remark}

We can now prove that \cref{cor:spectralrestatement} follows from \cref{thm:superpolytorsion}.
\begin{proof}
    Since $M$ has bounded geometry, there is a universal upper bound on $\lambda_1^1$ and $\lambda_0^1$.
    \begin{align*}
    h_0(M) &= \frac 2 {\diam M}\\
    h_1(M) &\geq \frac 1 {\frac{C\sqrt{\vol M}}{\sqrt{\lambda_1^1}} + D} &\text{by \cref{thm:cheeger}}\\
    &\geq \min\left(\frac {\sqrt {\lambda_1^1}} {C\sqrt{\vol M}}, 1/D\right)\\
    &\gtrsim \sqrt{\frac{\lambda_1^1}{\vol M}}\\
    h_2(M) &\geq \frac 1 {100}\min( \lambda_1^0, \sqrt{\lambda_1^0}) &\text{ by Buser's inequality}\\
    &\gtrsim \lambda_1^0\\
    \end{align*}
    Now plug into \cref{thm:superpolytorsion}.
\end{proof}

\section{Dehn surgery}\label{sec:dehn}
Our examples of expanders will be constructed using Dehn surgery. In this section, we explain how to put a metric of controlled volume and bounded geometry on the surgered manifold.

\begin{construction}[Dehn surgery]
Suppose $M$ is an oriented 3-manifold and $\gamma$ is a curve embedded in $M$. Let $N(\gamma)$ be a tubular neighbourhood of $\gamma$. A \emph{meridian} for $\gamma$, usually denoted $\mu$, is a simple closed curve in $\partial N(\gamma)$ which is contractible in $N(\gamma)$. The meridian is unique up to isotopy. A \emph{longitude}, usually denoted $\lambda$, is a simple closed curve in $\partial N(\gamma)$ which has intersection number 1 with a meridian. A \emph{framing} for $\gamma$ is a choice of a longitude. Given a framing for $\gamma$, the \emph{slope $q$} Dehn surgery is a 3-manifold denoted $M_\gamma(q)$ obtained by deleting $N(\gamma)$ and gluing in a new solid torus in such a way that the homology class $q\mu + \lambda$ is trivial in the new solid torus.
\end{construction}

\begin{construction}[Effective Dehn surgery]
Suppose $M$ is a Riemannian manifold with 1-bounded geometry. Suppose $\gamma$ is a closed curve in $M$ of length $\ell$ with geodesic curvature $\leq  1$ and a radius 1 embedded tubular neighbourhood. Call this tubular neighbourhood $N(\gamma)$. Suppose we choose a framing for $\gamma$ such that the longitude has a realization in $\partial N(\gamma)$ of length $\leq 10\ell+10$. The condition on its length simply prohibits longitudes that twist many times around $\gamma$. Then it is possible to perform a slope $q$ surgery such that metric on $M_\gamma(q)$ has 10-bounded geometry, agrees with $M$ outside of the tubular neighbourhood of $\gamma$, and the newly glued solid torus $T$ satisfies $\vol(T)\lesssim \poly(q,\ell)$ and $h_{1}(T,\partial T)\lesssim \poly(q,\ell)$.

To do this, first delete the tubular neighbourhood $N(\gamma)$ and add a collar of constant width so that the new boundary is totally geodesic and has the shape of a flat torus with dimensions $1\times \ell$, with the meridian $(1,0)$ and the longitude $(0,\ell)$. Call this metric $g_0$. 

Let $g_1$ be the pullback of $g_0$ by the action of $\begin{pmatrix} 1& q\\0& 1\end{pmatrix}$ on the torus. Let $g_t$ be the Teichmuller geodesic between $g_0$ and $g_1$, i.e. a 1-parameter family of metrics which are sheared parallel to the longitude. The length of this geodesic is about $q$. So the metric $g_t + dt^2$ on $T^2 \times [0,q]$ has bounded geometry and interpolates between $g_0$ on $T^2 \times 0$ and $g_1$ on $T^2 \times q$. Glue in the $T^2 \times [0,q]$ to our manifold, so that our manifold now has boundary isometric to $g_1$.

Finally, glue in a Euclidean solid torus whose boundary has shape $1\times \ell$. We can again do this with a small amount of smoothing at the boundary. This completes the Dehn surgery. The desired bounds $\vol(T)\lesssim \poly(q,\ell)$ and $h_{1}(T,\partial T) \lesssim \poly(q,\ell)$ follow from the fact that the metric on $T$ is $\poly(q,\ell)$-Lipschitz to a standard Euclidean solid torus of width 1 and height 1.
\end{construction}
As a warmup for the main argument, let us prove that large Dehn surgeries on links in $S^3$ have trivial rational homology in a hands on way.

\begin{prop}
    Suppose $L$ is a framed link in $S^3$. Perform simultaneous $q$ surgery on all the components of $L$. For large enough $q$, the result is a rational homology sphere.
\end{prop}
\begin{proof}
The linking matrix for $L$ is a presentation matrix for the first homology group of the Dehn surgery. For large enough $q$, the matrix is diagonally dominant, and therefore has nonzero determinant. Let's say the same thing in a more hands on way.

Let $\mu_1 \dots \mu_n$ be meridians for the components of $L$ and let $\lambda_1 \dots \lambda_n$ be our choices of longitudes. This means that each longitude $\lambda_i$ spans a punctured surface $\S_i$ in $S^3 \setminus L$ giving rise to the equation $\lambda_i \cong \sum_{j} lk(L_i, L_j) \mu_j$ in $H_1(S^3\setminus L)$. Here, $lk(L_i,L_i)$ should be interpreted as $lk(\lambda_i, L_i)$. Suppose $\gamma$ is a loop in $S^3 \setminus L$. Any such loop is homologous to a linear combination of meridians, $\sum_i a_i \mu_i$. In the Dehn surgered manifold $S^3_L(q)$, we have 

\begin{align}
\gamma &\cong \sum_i a_i \mu_i\\
&\cong \sum_i \frac{a_i}{q} \lambda_i\\
&\cong \sum_{i,j} \frac {a_i} q lk(L_i,L_j)\mu_j
\end{align}

For $q$ sufficiently large, we have $$\sum_{i,j} \left|\frac {a_i} q lk(L_i,L_j)\right| \leq \frac 1 2 \sum_i |a_i|.$$ So we have showed that every loop $\gamma$ is rationally homologous to a 1-cycle whose $L^1$ norm with respect to the basis $\mu_1 \dots \mu_n$ for $H_1(S^3\setminus K)$ is less than $\frac 1 2$ the $L^1$ norm of $\gamma$ with respect to the same basis. It follows that $\gamma$ is rationally null-homologous.
\end{proof}
The proof of the main theorem will follow a similar argument for a link in a connect sum of several $S^1 \times S^2$. A new difficulty in this setting is that the longitudes will not bound punctured surfaces as is the case for a link in $S^3$; we will show that the existence of some substitute surfaces for $\mathcal S_i$ is guaranteed by our assumptions on the coexpansion properties of an associated 2-complex.
\section{Promoting 2-complexes to 3-manifolds}\label{sec:promotecomplexes}

In this section, we construct 3-manifolds which are good expanders. We deduce \cref{thm:0} as \cref{cor:l2} and \cref{cor:l1} of the following theorem:

\begin{theorem}\label{thm:1}
    Suppose $X$ is a polyhedral 2-complex with local degree $k$.  Then there is a rational homology 3-sphere $\M(X,q)$ with a metric of $1$-bounded geometry such that the Cheeger constants of $\M(X,q)$ are controlled by Cheeger constants of $X$ as follows. Let $1\leq p \leq \infty$. Suppose 
    
    $$\max\left(\frac 1{h_0(X,\norm{\cdot}_p)}, \frac 1 {h_1(X,\norm{\cdot}_p)},\frac 1 {h^1(X,\norm{\cdot}_p)},\frac 1{\hprime(X,\norm{\cdot}_p)}\right)< C$$
    Then 
    
    $$\frac 1 {h_i(\M(X,q),\norm{\cdot}_p)} \leq \poly(\deg, C)$$
    for $i=0,1,2$.
\end{theorem}
Note that this theorem requires $X$ to have both good expansion and good coexpansion. While expansion and coexpansion are the same for manifolds by Poincar\'e duality, the same is not true for 2-complexes.

\begin{corollary}\label{cor:l2}
    There exists a sequence of rational homology spheres $M_i$ with metrics of 1-bounded geometry, $\vol(M_i)\to \infty$, and uniform lower bounds on the spectral gap for the Hodge Laplacian on $j$-forms for $j=0,1,2$.
\end{corollary}
\begin{proof}
    There exists a sequence $X_i$ of bounded degree 2-complexes with uniform lower bounds on $h_j(X_i,\norm{\cdot}_2)$ for all $j$. For example, quotients of some rank 2 Bruhat-Tits buildings are known to satisfy such bounds by Garland's method \cite{garland_p-adic_1973}. See \cite[Section 2.3]{lubotzky_high_2017} for a discussion of this construction in language closer to ours. As noted in \cref{prop:l2easy0}, control on $h_j(M_i, \norm{\cdot}_2)$ is the same as control on $h^{j+1}(X_i, \norm{\cdot}_2)$. Applying \cref{thm:1} to this sequence gives a sequence of rational homology spheres with all $L^2$ Cheeger constants uniformly bounded below. This is equivalent to control on the spectral gap of the Hodge Laplacian, as noted in \cref{prop:l2easy1}.
\end{proof}

The examples involving Bruhat-Tits buildings are not easy to understand. However, there are much simpler constructions that already give $L^1$ Cheeger constants of order $\frac 1 {\polylog(\vol(M_i))}$. The 2-skeleton of a $k$-dimensional hypercube has good expansion and coexpansion, as we now verify:

\begin{prop}
    Let $\mathcal{H}_{2,\deg}$ be the 2-skeleton of the $\deg$-dimensional hypercube. Then $\hprime(\mathcal{H}_{2,\deg},\norm\cdot_1) \gtrsim \frac 1 {\deg^2}$ and $h_1(\mathcal{H}_{2,\deg},\norm{\cdot}_1)\gtrsim \frac 1 \deg$. 
\end{prop}
\begin{proof}
    Let's first check that $h_1(\mathcal H_{2,\deg})\gtrsim \frac 1 \deg$. A 1-cycle of length $\ell$ passing through the origin can be described by a word of length $\ell$ on an alphabet of size $\deg$. We would like to find an efficent null-homotopy of this 1-cycle. A homotopy of this 1-cycle across a 2-cell corresponds with commuting two of the letters. Another kind of homotopy is the cancelling of a pair of adjacent letters which are the same. By the pigeonhole principle, there are two identical letters at distance $\deg$ in the word. So we can commute these letters to sit beside each other and cancel them for a cost of $\deg$. Therefore we have made progress in decreasing the norm of $c$, at a cost of $\deg$ elementary homotopies. So a 1-cycle of length $\ell$ can be contracted to a point in $\lesssim \deg \ell$ moves.

    Now let's check that $\hprime(\mathcal{H}_{2,\deg}) \gtrsim \frac 1 {\deg^2}$. For this part, it will be more convenient to work in the 3-skeleton of the hypercube, $\mathcal{H}_{3,\deg}$. This doesn't change $\hprime$. It suffices to show that any 2-cell $c$ has an efficient decomposition into a coexact and a closed piece. Define the discrete Laplacian $\Delta =d \partial + \partial d$. Consider its action on 2-chains of $\mathcal{H}_{3,\deg}$. Observe that $$\Delta c = (\deg+2)c -\sum_{c'||c} c',$$ where $c'||c$ means that $c'$ is parallel to $c$ and at distance 1 from $c$.

    Rearranging, we have
    
    \begin{align}
        c &= \frac 1 {\deg+2} (d \partial c + \partial d c) + \frac 1 {\deg+2} \sum_{c'||c} c'
    \end{align}
    Note that $\norm{\partial c}_1 = 4$ and $\norm{\partial d c}_1 \leq 6(\deg-2)$. So the first term above is already expressed as an efficient decomposition into a coexact and a closed piece. It remains to find an efficient decomposition of the last term. The last term is a 2-chain of norm $\leq \frac {\deg-2} {\deg+2}<1$, so we have made progress in decreasing the norm of $c$, for a cost of $\lesssim \deg$. Recursively repeating this procedure on the leftover 2-chain (just as in \cref{lem:expfill}), we get a sequence converging to the desired efficient decomposition of $c$. The size of the terms in the decomposition is $O(k^2)$, as desired.
\end{proof}

\begin{corollary}\label{cor:l1}
    There exists a sequence of rational homology spheres $M_i$ with metrics of 1-bounded geometry and $1/h_j(M_i, \norm{\cdot}_\mass, \R) \leq \polylog{\vol(M_i)}$ for $j=0,1,2$.
\end{corollary}
\begin{proof}
    Apply \cref{thm:1} to the 2-skeleton of a $k$-dimensional hypercube. This gives a 3-manifold with the desired control on $1/h_j(M_i, \norm{\cdot}_1,\R)$. Since $\norm{\cdot}_1$ and $\norm{\cdot}_\mass$ are comparable, we also get control on $1/h_j(M_i, \norm{\cdot}_\mass,\R)$
\end{proof}

\begin{construction}
First we explain the construction topologically; later we will choose the Riemannian metric. Given a polyhedral 2-complex $X$ and an integer $q$, we construct a closed 3-manifold $\M(X,q)$. Along the way, we will construct an intermediate 3-manifold $\M(X)$. 
\begin{figure}[h]
    \centering
    %% Creator: Inkscape 1.1.2 (0a00cf5339, 2022-02-04), www.inkscape.org
%% Accompanies image file 'M(X).pdf' (pdf, eps, ps)
\begingroup%
  \makeatletter%
  \providecommand\color[2][]{%
    \errmessage{(Inkscape) Color is used for the text in Inkscape, but the package 'color.sty' is not loaded}%
    \renewcommand\color[2][]{}%
  }%
  \providecommand\transparent[1]{%
    \errmessage{(Inkscape) Transparency is used (non-zero) for the text in Inkscape, but the package 'transparent.sty' is not loaded}%
    \renewcommand\transparent[1]{}%
  }%
  \providecommand\rotatebox[2]{#2}%
  \newcommand*\fsize{\dimexpr\f@size pt\relax}%
  \newcommand*\lineheight[1]{\fontsize{\fsize}{#1\fsize}\selectfont}%
  \ifx\svgwidth\undefined%
    \setlength{\unitlength}{288bp}%
    \ifx\svgscale\undefined%
      \relax%
    \else%
      \setlength{\unitlength}{\unitlength * \real{\svgscale}}%
    \fi%
  \else%
    \setlength{\unitlength}{\svgwidth}%
  \fi%
  \global\let\svgwidth\undefined%
  \global\let\svgscale\undefined%
  \makeatother%
  \begin{picture}(1,0.75)%
    \lineheight{1}%
    \setlength\tabcolsep{0pt}%
    \put(0,0){\includegraphics[width=\unitlength,page=1]{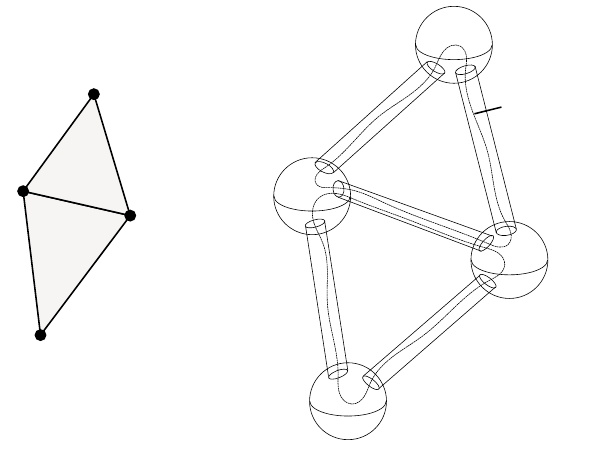}}%
    \put(0.13837991,0.46068468){\color[rgb]{0,0,0}\makebox(0,0)[t]{\lineheight{1.25}\smash{\begin{tabular}[t]{c}$f$\end{tabular}}}}%
    \put(0.086447,0.18061693){\color[rgb]{0,0,0}\makebox(0,0)[lt]{\lineheight{1.25}\smash{\begin{tabular}[t]{l}$v$\end{tabular}}}}%
    \put(0.09269677,0.52107529){\color[rgb]{0,0,0}\makebox(0,0)[rt]{\lineheight{1.25}\smash{\begin{tabular}[t]{r}$e$\end{tabular}}}}%
    \put(0.61808747,0.57302846){\color[rgb]{0,0,0}\makebox(0,0)[rt]{\lineheight{1.25}\smash{\begin{tabular}[t]{r}$\M(e)$\end{tabular}}}}%
    \put(0.84202481,0.55891385){\color[rgb]{0,0,0}\makebox(0,0)[lt]{\lineheight{1.25}\smash{\begin{tabular}[t]{l}$\M(f)$\end{tabular}}}}%
    \put(0.65322583,0.05496696){\color[rgb]{0,0,0}\makebox(0,0)[lt]{\lineheight{1.25}\smash{\begin{tabular}[t]{l}$\M(v)$\end{tabular}}}}%
  \end{picture}%
\endgroup%

    \caption{$X$ and $\M(X)$}
    \label{fig:M(X)}
\end{figure}
For each vertex $v$ of $X$, make a 3-sphere $\M(v)$. For each edge $e$ of $X$, add a connect sum tube (an $S^2\times I$) $\M(e)$  between the 3-spheres corresponding with its endpoints. Call the resulting 3-manifold $\M(X)$. $\M(X)$ is homeomorphic to a connect sum of many copies of $S^1 \times S^2$. For each 2-cell $f$ of $X$, embed a closed curve $\M(f)$ in $\M(X)$ which traverses the tubes corresponding with the edges in $\partial f$. At this step, there is some choice in how the various $\M(f_i)$ link with one another.  Let $N(\M(f))$ be a tubular neighbourhood of $\M(f)$. Choose a framing on $\M(f)$, and let $\lambda(f)$ (resp $\mu(f)$) be a curve in $\partial N(\M(f))$ realizing the longitude (resp meridian). Do Dehn surgery on $\M(f)$ with slope $q$. Call the resulting 3-manifold $\M(X,q)$. Let $T(f)$ be the new solid torus which is glued in during the Dehn surgery.
\end{construction}

Of course there is a lot of choice in this construction. For each 2-cell $f$, we need to choose how to embed $\M(f)$ and we need to choose a longitude for $\M(f)$. We also need to choose a Riemannian metric. These choices can all be made so that the resulting metric has $\poly(k)$-bounded geometry and distances increase compared to $X$ by a factor of $O(\poly(\deg))$.

\begin{prop}
    Suppose the local degree of $X$ is $k$.  We can choose the metric on $\M(X)$ so that the following hold:
    \begin{enumerate}
        \item For each vertex $v_i$ of $X$, $\M(v_i)$ is round 3-sphere with radius $\poly(\deg)$. As a result, $h_1(\M(v_i))\gtrsim \frac 1 {\poly(\deg)}$.
        \item For each edge $e_i$ of $X$, $\M(e_i)$ is a tube with radius $poly(\deg)$ and length 1. As a result, $h_1(\M(e_i)) \gtrsim \frac 1 {\poly(\deg)}$
        \item For each face $f_i$ of $X$, $\M(f_i)$ has an embedded tubular neighbourhood of radius 1, length $\poly(\deg)$, geodesic curvatures $\leq 1$, and a choice of longitude on $\partial N(\M(f_i))$ with length $\leq 10 |\M(f_i)|$. Furthermore, all these tubular neighbourhoods are disjoint. 
        \item The Dehn surgeries are performed so that $\M(X,q)$ has sectional curvatures $\leq 1$ and volume $\poly(q,\deg) |X|$. Furthermore, $h_1(T(f_i), \partial T(f_i))\gtrsim \frac 1 {\poly(\deg,q)}$.
    \end{enumerate}
\end{prop}
\begin{proof}
    For each vertex $v$, there are at most $O(\deg)$ tubes $N(\M(f_i))$ that we need to route through $\M(v)$. Similarly, for each edge $e$, there are $O(\deg)$ tubes we need to route through $\M(e)$. So $\poly(\deg)$ volume is more than enough to route all of these tubes without intersection. Control on the Dehn surgeries comes from our discussion of effective Dehn surgery in the previous section.
\end{proof}

To bound the Cheeger constants of $\M(X,q)$, we will use a triangulation of $\M(X,q)$ having bounded geometry. Choose once and for all such a triangulation of bounded geometry for $\M(X,q)$. Choose the triangulation fine enough that $\partial N(T(f_i))$ and the interfaces between $\cup_i\sigma(v_i)$ and $\cup_i\sigma(e_i)$ can be realized as disjoint simplicial surfaces. In what follows, a \emph{simplicial curve or surface} in $\M(X,q)$ is a 1-chain or a 2-chain in this triangulation. A \emph{simple normal surface} in $\M(X,q)$ is a surface transverse to the triangulation of $\M(X,q)$ whose intersection with any given tetrahedron is either a triangle (separating one vertex from the other three) or a square (separating two vertices from the other two).

\begin{lemma}[Simplicial splitting lemma]
Suppose $\Sigma$ is a normal surface in a triangulated manifold whose intersection with each tetrahedron has at most one connected component. Let $c$ be a simplicial 1-cycle having zero algebraic intersection number with $\Sigma$. Then $c$ can be transported in $N(\Sigma)$, for cost $diam(\Sigma)area(\Sigma)$, to a simplicial 1-cycle $c'$ whose support does not include any edge intersecting $\Sigma$. Here, the costs are measured using $\norm{\cdot}_p$.
\end{lemma}
\begin{proof}
The idea is basically to rewire $c$ to avoid $\Sigma$. As a simple normal surface, $\Sigma$ inherits a cell decomposition into triangles and quadrilaterals. Each of the edges of this cell decomposition corresponds with a triangle in the triangulation of $M$. Let $c_0$ be the 0-cycle $\Sigma \cap c$. We may write $c_0 = \sum_i c_i$ such that $\norm{c_0}_1 = \sum_i \norm{c_i}_1$ and $c_i$ is supported on two points with opposite sign. For each $c_i$, find an embedded 1-chain $t_i$ connecting its two endpoints. Let $T_i$ be the simplicial 2-chain obtained by replacing each 1-cell of $t_i$ with the unique triangle of the triangulation of $\sigma(X,q)$ containing it; this ensures that $T_i \cap \Sigma = t_i$. Then $c-\sum_i T_i$ has support disjoint from $\Sigma$. We also have

\begin{align*}
    \norm{\sum_i T_i}_p &\leq \sum_i \norm{T_i}_p\\
    &\leq \sum_i \norm{T_i}_1\\
    &= \sum_i \norm{t_i}_1\\
    &\leq \sum_i \norm{c_i}_1 \diam(\Sigma)\\
    &\leq \norm{c_0}_1 \diam(\Sigma)\\
    &\leq \norm{c}_p \area(\Sigma)^{1-1/p} \diam(\Sigma)
\end{align*}

\end{proof}

Here are a few kinds of simplicial surfaces with boundary which will be used to transport curves around the manifold.
\begin{enumerate}
    \item[(Type 1)] For each $i$, there is a simplicial 2-chain in the Dehn surgery 2-handle $T(f_i)$ with boundary $q \mu(f_i) + \lambda(f_i)$. 
    \item[(Type 2)] The surfaces described below in \cref{lem:nullhom}. See \cref{fig:surface2}.
    \item[(Type 3)] The restriction of a cross-section of $\M(e_i)$ to $\M(X)\setminus \cup_i N(\M(f_i))$. It is topologically a sphere with several disks removed. This surface can be viewed in $\M(X)$ or in $\M(X,q)$. See \cref{fig:surface3}.
    \item[(Type 4)] Surfaces in $\partial T(f_i)$.
\end{enumerate}

\begin{figure}
    \centering
    \def\svgwidth{3in}
    %% Creator: Inkscape 1.1.2 (0a00cf5339, 2022-02-04), www.inkscape.org
%% Accompanies image file 'surface3.pdf' (pdf, eps, ps)
\begingroup%
  \makeatletter%
  \providecommand\color[2][]{%
    \errmessage{(Inkscape) Color is used for the text in Inkscape, but the package 'color.sty' is not loaded}%
    \renewcommand\color[2][]{}%
  }%
  \providecommand\transparent[1]{%
    \errmessage{(Inkscape) Transparency is used (non-zero) for the text in Inkscape, but the package 'transparent.sty' is not loaded}%
    \renewcommand\transparent[1]{}%
  }%
  \providecommand\rotatebox[2]{#2}%
  \newcommand*\fsize{\dimexpr\f@size pt\relax}%
  \newcommand*\lineheight[1]{\fontsize{\fsize}{#1\fsize}\selectfont}%
  \ifx\svgwidth\undefined%
    \setlength{\unitlength}{288bp}%
    \ifx\svgscale\undefined%
      \relax%
    \else%
      \setlength{\unitlength}{\unitlength * \real{\svgscale}}%
    \fi%
  \else%
    \setlength{\unitlength}{\svgwidth}%
  \fi%
  \global\let\svgwidth\undefined%
  \global\let\svgscale\undefined%
  \makeatother%
  \begin{picture}(1,0.75)%
    \lineheight{1}%
    \setlength\tabcolsep{0pt}%
    \put(0,0){\includegraphics[width=\unitlength,page=1]{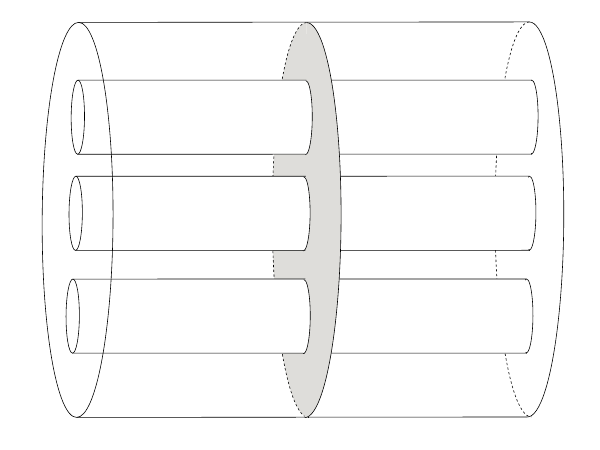}}%
    \put(0.71818231,0.54992664){\color[rgb]{0,0,0}\makebox(0,0)[t]{\lineheight{1.25}\smash{\begin{tabular}[t]{c}$N(\M(f_1))$\end{tabular}}}}%
    \put(0.72297802,0.3846892){\color[rgb]{0,0,0}\makebox(0,0)[t]{\lineheight{1.25}\smash{\begin{tabular}[t]{c}$N(\M(f_2))$\end{tabular}}}}%
    \put(0.71494703,0.21462766){\color[rgb]{0,0,0}\makebox(0,0)[t]{\lineheight{1.25}\smash{\begin{tabular}[t]{c}$N(\M(f_3))$\end{tabular}}}}%
    \put(0,0){\includegraphics[width=\unitlength,page=2]{surface3.pdf}}%
    \put(0.65008974,0.08000384){\color[rgb]{0,0,0}\makebox(0,0)[lt]{\lineheight{1.25}\smash{\begin{tabular}[t]{l}$\S_3$\end{tabular}}}}%
  \end{picture}%
\endgroup%

    \caption{The picture shows $\M(e)$ for some edge $e$, which is topologically an $S^2\times [0,1]$. A type 3 surface is shown in grey.}
    \label{fig:surface3}
\end{figure}

\begin{figure}
    \centering
    %% Creator: Inkscape 1.1.2 (0a00cf5339, 2022-02-04), www.inkscape.org
%% Accompanies image file 'surface2.pdf' (pdf, eps, ps)
\begingroup%
  \makeatletter%
  \providecommand\color[2][]{%
    \errmessage{(Inkscape) Color is used for the text in Inkscape, but the package 'color.sty' is not loaded}%
    \renewcommand\color[2][]{}%
  }%
  \providecommand\transparent[1]{%
    \errmessage{(Inkscape) Transparency is used (non-zero) for the text in Inkscape, but the package 'transparent.sty' is not loaded}%
    \renewcommand\transparent[1]{}%
  }%
  \providecommand\rotatebox[2]{#2}%
  \newcommand*\fsize{\dimexpr\f@size pt\relax}%
  \newcommand*\lineheight[1]{\fontsize{\fsize}{#1\fsize}\selectfont}%
  \ifx\svgwidth\undefined%
    \setlength{\unitlength}{288bp}%
    \ifx\svgscale\undefined%
      \relax%
    \else%
      \setlength{\unitlength}{\unitlength * \real{\svgscale}}%
    \fi%
  \else%
    \setlength{\unitlength}{\svgwidth}%
  \fi%
  \global\let\svgwidth\undefined%
  \global\let\svgscale\undefined%
  \makeatother%
  \begin{picture}(1,0.75)%
    \lineheight{1}%
    \setlength\tabcolsep{0pt}%
    \put(0,0){\includegraphics[width=\unitlength,page=1]{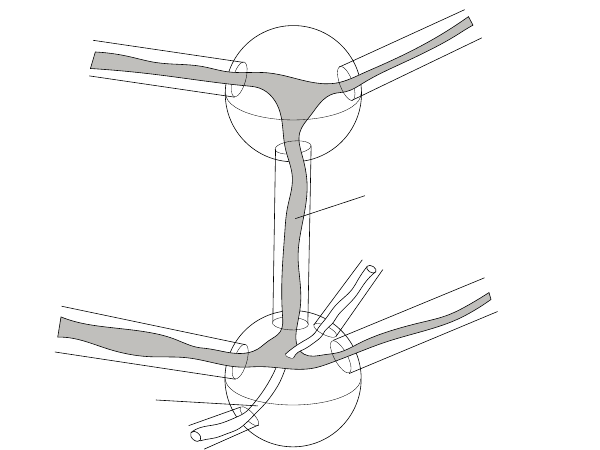}}%
    \put(0.61364722,0.4143749){\color[rgb]{0,0,0}\makebox(0,0)[lt]{\lineheight{1.25}\smash{\begin{tabular}[t]{l}$\mathcal S_2$\end{tabular}}}}%
    \put(0.25528289,0.07392346){\color[rgb]{0,0,0}\makebox(0,0)[rt]{\lineheight{1.25}\smash{\begin{tabular}[t]{r}$N(\M(f_i))$\end{tabular}}}}%
  \end{picture}%
\endgroup%

    \caption{A surface of type 2, as constructed in \cref{lem:nullhom}}
    \label{fig:surface2}
\end{figure}
\begin{lemma}\label{lem:nullhom}
    If $c$ is a null-homologous simplicial 1-cycle in $\M(X)$, then $c$ bounds a simplicial 2-chain in $\M(X)$ of norm $\lesssim \poly(\deg)\norm{c}$. Moreover, if $c\subset \M(X) \setminus \cup_i N(\M(f_i))$, then $c$ can be transported in $\M(X) \setminus \cup_i N(\M(f_i))$ to a linear combination of meridians $\sum_i x_i \mu(f_i)$ of norm $\lesssim \poly(\deg) \norm{c}$, for $\poly(\deg)$ cost.
\end{lemma}
\begin{proof}
    Given a simplicial 1-cycle $c$ in $\M(X)$, use the simplicial splitting lemma on the interface between $\M(e)$ and $\M(v)$ for each incidence between an edge $e$ and a vertex $v$ of $X$. This gives a transport of $\poly(\deg)$ cost to a new curve $c'$ each of whose components lives in $\M(v)$ for some vertex $v$ or $\M(e)$ for some edge $e$. The we can fill $c'$ for $\poly(\deg)$ cost because $\M(v)$ and $\M(e)$ both have filling constants bounded by $\poly(\deg)$. This gives the desired filling of $c$. Call this filling $S$.

    To find the desired transport to a collection of meridians, the idea is to delete $S\cap \cup_i N(\M(f_i))$ from $S$. The resulting surface is a transport from $c$ to a simplicial 1-cycle in $\cup_i \partial N(\M(f_i))$ which is homologous to a collection of meridians. Then use the fact that $h_1(\partial T,\norm{\cdot}_p)\gtrsim \frac 1 {\poly(k)}$ to efficiently transport this 1-cycle to a linear combination of meridians for cost $\poly(k)$.
\end{proof}

\begin{lemma}\label{lem:fillcyc}
    If $\gamma$ is a simplicial 1-cycle in $\M(X,q)$, then $\gamma$ may be transported to a linear combination of meridians of norm $\lesssim \frac 1 {h_1(X,\norm \cdot)_p}\poly(q,\deg)\norm{\gamma}_p$ for cost $\poly(q,\deg)$.
\end{lemma}
\begin{proof}
    First, we transport $\gamma$ to a simplicial 1-cycle $\gamma'$ which sits in the complement of the Dehn surgery handles, ie inside $\M(X,q)\setminus \cup_i T(f_i)$. This uses the fact that for each 2-cell $f$, the corresponding Dehn surgery solid torus $T(f)$ satisfies $1/h_{1}(T(f),\partial T(f))\lesssim \poly(\deg,q)$. This allows us to transport $\gamma$ outside of $T(f)$ for cost $\poly(\deg,q)$.

    Let $\sigma^{-1}(\gamma')$ be the image of $\gamma'$ under the obvious projection $\M(X)\to X$. Note that $\norm{\M^{-1}(\gamma')}_p\lesssim \norm\gamma_p$. Using the fact that $X$ is a good expander in dimension 1, find an efficient 2-chain $C=\sum_i a_i f_i$ in $X$ satisfying $\partial C = \M^{-1}(\gamma')$. Efficient means that $\norm{C}_p\leq \frac {\norm{\M^{-1}(\gamma')}_p} {h_1(X,\norm{\cdot}_p)}$. For $\poly(\deg,q)$ cost, transport $\gamma'$ to $\gamma'' := \gamma' +  \sum_i a_i (q\mu_i + \lambda_i)$. This transport uses type 1 surfaces in $T(f_i)$ for each $i$ such that $a_i\neq 0$. The size of this transport is $\poly(\deg,q)\norm{C}_p$, so this transport is also efficient.

    Now think of $\gamma''$ as a 1-cycle in $\M(X)$. Note that $\gamma''$ is null-homologous in $\M(X)\setminus \cup_i T(f_i)$. So we can fill $\gamma''$ in $\M(X)$ for cost $\poly(\deg,q)$. The filling 2-chain $S$ has norm $\poly(\deg,q)\norm{\gamma''}_p\lesssim \poly(\deg,q)\norm{\gamma}_p$. As before, delete $S\cap \cup_i N(\M(f_i))$ from $S$ to obtain to get a surface which transports $S$ to a linear combination of meridians. By \cref{prop:dbounded}, this linear combination of meridians has norm $\lesssim \norm{S}_p$. Therefore, we have established that $\gamma$ can be transported to the sum of meridians of norm $\lesssim \poly(q,\deg)\norm{\gamma}_p$ for $\poly(\deg,q)$ cost.
\end{proof}

\begin{lemma}\label{lem:fillmerid}
    Let $\gamma=\sum_i z_i \mu(f_i)$ be a linear combination of meridians. Then $\gamma$ can be transported to a linear combination of meridians of norm $\leq \frac 1 {q\hprime(X,\norm{\cdot}_p)}\poly(\deg) \norm{\gamma}_p$ for cost $\frac 1 {\hprime(X)}\poly(q,\deg)$.
\end{lemma}

\begin{proof}
    Think of $\gamma$ as corresponding to the 2-chain $\sum_i z_i f_i$ in $X$. We will freely the correspondence between 2-chains and 2-cochains in $X$. By the definition of $\hprime(X)$, there is a 1-chain $x=\sum_i x_i e_i$ and a closed 2-chain $y=\sum_i y_i f_i$ satisfying $\sum_i z_i f_i = d x + y$ and $\norm{x}_p + \norm{y}_p\leq \frac 1 {\hprime(X)}\norm{\sum_i z_i f_i}_p$. 

    Now the type 3 surface $\sum_i x_i \S_{3,i}$ transports $\gamma$ to $\sum_i y_i \mu(f_i)$ for cost $\poly(\deg)$. The type 1 surface $\sum_i y_i \S_{1,i}$ transports $\sum_i y_i \mu(f_i)$ to $\sum_i y_i \frac 1 q \lambda_i$ for cost $\poly(q,\deg)$. Since $y$ was a closed 2-chain, $\sum_i y_i \frac 1 q \lambda_i$ is null-homologous in $\M(X)$. By \cref{lem:nullhom}, we may now transport $\sum_i y_i \frac 1 q \lambda_i$ to a linear combination of meridians of norm $\frac 1 q\poly(\deg) \norm{y}_p \lesssim \frac 1 {q\hprime(X)} \poly(\deg) \norm{\gamma}_p$ for $\poly(k)$ cost. The total cost of this transport is $\frac 1 {\hprime(X,\norm{\cdot}_p)} \poly(q,\deg)$.
\end{proof}

Now we are ready to prove \cref{thm:1}.

\begin{proof}[Proof of \cref{thm:1}]
    Recall that $C$ is the name for an upper bound on the reciprocals of all the Cheeger constants of $X$. If $q$ is a sufficiently high degree polynomial in $\deg$ and $C$, then by \cref{lem:fillmerid}, every linear combination of meridians can be transported for $\poly(\deg,C)$ cost to a linear combination of meridians which has half the norm. By \cref{lem:expfill}, it follows that every linear combination of meridians can be filled for $\poly(\deg,C)$ cost. By \cref{lem:fillcyc}, any 1-cycle can be transported to a linear combination of meridians of controlled norm for $\poly(q,\deg)$ cost. So it follows that any 1-cycle can be filled for $\poly(\deg,C)$ cost. 

    With $q$ now fixed, there is a surjection $\pi$ from the 1-skeleton of the triangulation of $\M(X,q)$ to the 1-skeleton of $X$ whose fibers are connected and have size $\poly(\deg,C)$. Since the fibers of $\pi$ have size $\poly(\deg, C)$, their Cheeger constants are also $O(\poly(\deg,C))$. Applying \cref{lem:lerayserre} to this map, we find that 
$$    \frac 1 {h_0(\M(X,q),\norm{\cdot}_p)} \leq \poly(\deg, C)
$$
and
$$    \frac 1 {h^1(\M(X,q),\norm{\cdot}_p)} \leq \poly(\deg, C)
$$
So far, all these Cheeger constants are simplicial Cheeger constants. But by \cref{prop:simplicial}, the same inequalities hold for Riemannian Cheeger constants. By Poincar\'e duality, $h_2(\M(X,q),\norm{\cdot}_p) = h^1(\sigma(X,q),\norm{\cdot}_p)$. So we have established the desired control on all the Cheeger constants of $\M(X,q)$.
\end{proof}

\section{Torsion homology and topological overlap}\label{sec:topoverlap}
In this section, we prove \cref{thm:superpolytorsion} and \cref{thm:3}. All norms in this section are mass norms. We will also use polyhedral chains and rational coefficients everywhere as permitted by \cref{prop:realrational}. The arguments are variations of Gromov's proof of his topological overlap theorem. The original argument works with $\Z/2$ or $\Z$ coefficients, and rests on the fact that the minimal norm of a representative of a nonzero multiple of the fundamental class has norm $\sim \vol(M)$. The following trick relates rational filling to integer filling:

\begin{lemma}\label{lem:ratint}
Suppose $M$ is an oriented 3-manifold and $\gamma$ is a 1-cycle with integer coefficients which is trivial in integer homology. If $\gamma$ can be filled with a rational 2-chain of area $A$, then $\gamma$ can be filled with an integer 2-chain of area $A$.    
\end{lemma}
\begin{proof}[Proof sketch]
    Let $S$  be the rational 2-chain filling $\gamma$. Then for some $n$, $nS$ is an integer 2-chain filling $n\gamma$ having area $nA$. Approximate $nS$ by an immersed surface, then cut and paste to remove the self-intersections. The resulting surface $R$ is embedded except possibly along its boundary, and has area $nA + \varepsilon$. This surface is Poincar\'e dual an element of $n\cdot H^1(M\setminus \gamma,\Z)$. Therefore, the surface must cut $M\setminus \gamma$ into $n$ pieces. Therefore $R$ has $n$ components, each of which fills $\gamma$. One of these components must have area $\leq A+\varepsilon/n$. Now we may take $\varepsilon\to 0$.
\end{proof}
    This lemma means that rational filling for 1-cycles in 3-manifolds is the same as filling, as long as the loops are trivial in homology. The same thing is true for filling surfaces:

\begin{lemma}\label{lem:ratint2}
    Suppose $\Sigma$ is a null-homologous integer 2-cycle on a 3-manifold $M$. If $\Sigma$ can be filled with a rational 3-chain of volume $A$, then $\Sigma$ can also be filled with an integer 3-chain of volume $A$.
\end{lemma}
\begin{proof}
Let $x_1 \dots x_n$ be the connected components of $M \setminus \Sigma$. A rational 3-chain $A$ filling $\Sigma$ may be expressed as $A=\sum_i a_i x_i$ for some rational coefficients $a_i$. Since $\Sigma$ is an integer 2-chain, the fractional part of $a_i$ is the same for neighbouring regions, and is therefore independent of $i$. We can add any multiple of the fundamental class of $M$ to obtain a new filling of $\Sigma$. In other words, for any $t\in \R$, $A(t):=\sum_i (\lfloor a_i\rfloor+t) x_i$ is also a filling of $\Sigma$. Note that $\norm {A(t)}$ is a linear function on the interval $[0,1]$. Therefore, the minimal $L^1$ norm of a filling on this interval is attained at $t=0$ or $t=1$, ie at an integer filling.
\end{proof}

\subsection{Warmup}
Before proving \cref{thm:superpolytorsion}, we prove an easier version which gives linear growth for $|\Ht{M}|$ with respect to volume. It has the advantage that the constants can be made completely explicit.
\begin{theorem}\label{thm:2}
Let $M$ be a rational homology 3-sphere of 1-bounded geometry and let $N$ be the least common multiple of the orders of elements of $H_1(M)$. There exists a universal constant $c$ such that
    $$  N > c\,  h_0(M) h_1(M)  h_2(M)  \vol(M)$$
\end{theorem}
\begin{proof}[Proof of \cref{thm:2}]
    Let $M$ be a rational homology sphere. Let $N$ be the least common multiple of the orders of the elements of $H_1(M,\Z)$. For any 1-chain $\gamma$ in $M$, $n\gamma$ is null-homologous. The proof of the theorem follows Gromov's proof, except that we use chains with coefficients in $\frac 1 N \Z$.
    
    Choose a very fine triangulation of $M$, so that the volume of any simplex is $\leq \varepsilon$. Choose a point $p$. We will try to define a chain contraction $H$ of the simplicial chain complex to $p$. We do this inductively on the $i$-skeleta of the triangulation, controlling the norm of $H$ at each stage.
    
    For every $0$-cell $q$ in the triangulation, find a path of length $\frac 1 {h_0(M)}$ from $p$ to $q$. Define $H(q)$ to be this path. We have $$\norm{H(q)} \leq \frac 1 {h_0(M)}$$ for all $0$-chains $q$.
    
    For every edge $e$ of the triangulation, $H\partial e + e$ is a 1-cycle. The 1-cycle $N(H\partial e + e)$ is null-homologous, so can be filled an integer 2-chain. By \cref{lem:ratint}, $H\partial e + e$ can be filled with a rational 2-chain having coefficients in $\frac 1 N \Z$. Define $He$ to be this rational 2-chain. We have 
    \begin{align}
    \norm{He} &\leq \left(2\frac 1 {h_0(M)} \frac 1 {h_1(M)}+\varepsilon\right)\\
    &\leq 3\frac 1 {h_0(M)} \frac 1 {h_1(M)}+O(\varepsilon)
    \end{align}
    for all $1$-chains $e$.
    
    For every triangle $t$ in the triangulation, $H\partial t + t$ is a 2-cycle with coefficients in $\frac 1 N \Z$. We have 
    \begin{align}
    \norm{Ht} &\leq \frac 1 {h_2(M)} \norm{H\partial t + t}\\
    &\leq \frac 3 {h_2(M)}\left(3\frac 1 {h_0(M)} \frac 1 {h_1(M)}\right)+O(\varepsilon)\\
    &\leq 10\frac 1 {h_2(M)}\frac 1 {h_0(M)} \frac 1 {h_1(M)}+ O(\varepsilon)
    \end{align}
    
    By \cref{lem:ratint2}, this 2-cycle can be filled with a rational 2-chain with coefficients in $\frac 1 N \Z$. For every tetrahedron $r$ in the triangulation, $H \partial r + r$ is a 3-cycle with coefficients in $\frac 1 N \Z$. Suppose this 3-cycle has norm less than $\vol(M)/N$. Then it cannot represent a nontrivial multiple of the fundamental class of $M$. Therefore, it must be the trivial 3-cycle. So we can define $Hr=0$. If this is the case for every tetrahedron $r$, then we have completed our chain homotopy of $M$ to a point, which is a contradiction because $H_3(M)\neq 0$. So there must have been some $r$ with such that $\norm{H\partial r + r} \geq \vol(M)/N$. For this choice of $r$, we now have
    \begin{align}
    \vol(M)/N &\leq \norm{H\partial r + r}\\
        &\leq 40\frac 1 {h_2(M)}\frac 1 {h_0(M)} \frac 1 {h_1(M)} + O(\varepsilon)
    \end{align}
    This desired inequality is obtained by clearing denominators in the inequality above and sending $\varepsilon$ to 0.
\end{proof}

\subsection{Diameter bounds for the universal abelian cover}
The goal of this subsection is to prove the following lemma:

\begin{lemma}\label{lem:torsdiameter}
For any metric space $M$ with $H_1(M,\Z)$ finite, the diameter of its universal abelian cover $\widetilde{M}$ is controlled as follows: $$\diam(\widetilde{M}) \lesssim_{\delta} |H_1(M,\Z)|^{\delta} \diam(M)$$ for any $\delta>0$.
\end{lemma}

The proof crucially uses the following result of Benjamini, Finucane, and Tessera:
\begin{theorem}[Main result of \cite{benjamini_scaling_2014}]\label{thm:torusconvergence}
Let $(G_i)$ be an unbounded sequence of finite, connected, vertex transitive graphs such that $|G_i| = O(diam(G_i)^\delta)$ for some $\delta > 0$. Up to taking a subsequence, and after rescaling by the diameter, the sequence $G_i$ converges in the Gromov Hausdorff distance to some finite dimensional torus equipped with some invariant Finsler metric. (Recall that a flat Finsler metric on $\R^n$ is a path metric induced by a norm on $\R^n$, and an invariant Finsler metric on a torus is the quotient of a flat Finsler $\R^n$ by a lattice.)
\end{theorem}

Before giving the proof of \cref{lem:torsdiameter}, let us fix some language for describing coarse isometries. Suppose $X$ and $Y$ are metric spaces. Say that two maps $f,g:X\to Y$ are $\varepsilon$-close if $d(f(x),g(x))<\varepsilon$ for all $x\in X$. We say that $\rho:X\to Y$ is an $\varepsilon$-isometric embedding if $|d(x,y) - d(\rho(x),\rho(y))| < \varepsilon$ for all $x,y\in X$. We say $\rho$ is $\varepsilon$-surjective if every point of $Y$ is within distance $\varepsilon$ of $\rho(X)$. We say that $\rho$ is an $\varepsilon$-isometry if $\rho$ is an $\varepsilon$-surjective isometric embedding. We say that two maps $f,g: X\to X$ $\varepsilon$-commute if $fg$ is $\varepsilon$-close to $gf$. The following three lemmas are rather believable; the impatient reader can skip ahead to the proof of \cref{lem:torsdiameter} and refer back to them as necessary.

\begin{lemma}\label{lem:isomembed}
    Suppose $Y$ is a compact metric space such that every ball of radius $<100\varepsilon$ is contractible. Let $X_1$ be a compact metric space. Suppose $f:X_1\to Y$ is an $\varepsilon$-isometry. For any metric space $X_2$ homeomorphic to a $k$-dimensional simplicial complex and having Gromov-Hausdorff distance $\varepsilon/k$ from $X_1$, there is a $10\varepsilon$-isometric inclusion $f':X_2 \to Y$.
\end{lemma}
\begin{proof}
    Choose a triangulation of $X_2$ such that the diameter of any simplex is $\ll \varepsilon$. Since $X_1$ and $X_2$ are close enough in Gromov-Hausdorff distance, one can map the vertices of $X_2$ to $Y$ in a way that respects distances up to an additive error of $\varepsilon$. Since $Y$ is locally contractible, we can extend this map to the higher simplices of $X_2$. Moreover, when extending the map across the $i+1$-skeleton, we worsen the additive distortion of distances by at most $5\varepsilon/k$. After extending across all the simplices, the resulting map distorts distances by at most $k\cdot 5 \varepsilon/k + \varepsilon < 10\varepsilon$.
\end{proof}

\begin{lemma}\label{lem:commute}
    Suppose $\rho:T^k \to T^k$ is an affine isometry of flat finite dimensional torus to itself. Let $A$ be the set of translations which $\varepsilon$-commute with $\rho$. If $\rho$ is not itself a translation, then
    $$ \vol(A) =O(\varepsilon \vol(T^k))$$
    where the implied constants may depend on the geometry of $T^k$.
\end{lemma}
\begin{proof}
    The set of translations which commute on the nose with $\rho$ has codimension at least one. We would like to upgrade this to the assertion that the set of translations which $\varepsilon$-commute with $\rho$ has small measure. Choose a linear coordinate system which identifies $T^k$ with $\R^k/\Z^k$. Let $\norm{\cdot}$ be the invariant Finsler metric. Write $\rho$ in coordinates as $\rho(x)=Rx + w$, where $R$ is a $k\times k$ matrix and $w$ is a vector. Let $v$ be a vector defining a translation in $A$. Without loss of generality, assume that $v\in [0,1]^k$. Let $N=\sup_{v\in [0,1]^k}\norm{v}$. The condition that $\rho$ and $v$ are $\varepsilon$-commuting maps is
    $$\min_{p\in \Z^k} \norm{Rv - v-p}<\varepsilon$$
    Since $R$ is a $\varepsilon$-isometry, $$\norm{Rv-v} < \norm{Rv} + \norm{v} < 2\norm{v} + \varepsilon < 3N.$$ Therefore, we can safely replace $\min_{p\in \Z^k}$ with $\min_{p\in \Z^k, \norm{p} < 10N}$. It follows that

    \begin{equation}\label{eqn:union}
        A \subseteq \bigcup_{p\in \Z^k, \norm{p}<10N} \{  v \in [0,1]^k \mid \norm{Rv - v -p} < \varepsilon \}
    \end{equation}

    Since $R$ is not the identity, the solutions to $Rv - v - p=0$ are a proper subspace of $\R^n$. The solutions to $\norm{Rv - v -p} < \varepsilon $ are a radius $ \varepsilon / \norm{R-I}$ neighbourhood of this subspace. There are only finitely many possibilities for $R$ once the geometry of $T^k$ is fixed, so $\norm{R-I}$ is bounded below. Finally, 
    \begin{align*}
    \vol\{  v \in [0,1]^k \mid \norm{Rv - v -p} < \varepsilon \} &= O(\varepsilon \vol([0,1]^k))\\
    &= O(\varepsilon \vol(T^k))
    \end{align*}
    Since $A$ is a finite union of such sets, $\vol(A) = O(\varepsilon \vol(T^k))$.
\end{proof}

\newcommand{\Tl}{T^k}

\begin{lemma}\label{lem:affine}
    Suppose $T^k$ is a $k$-dimensional torus with a flat Finsler metric. Suppose $\varepsilon$ is much smaller than the injectivity radius of $T^k$. Then any $\varepsilon$-isometry $f$ of $T^k$ to itself is $10\varepsilon$-close to an affine $\varepsilon$-isometry $f'$.
\end{lemma}
\begin{proof}
    One can find $f'$ using a standard averaging argument: the average of $f$ with all of its conjugates by translation in $\Tl$ is a linear map with the desired properties.
\end{proof}
Note that in \cref{lem:affine}, we cannot guarantee that $f'$ is actually an affine isometry. For example, $f$ might be a reflection on a 2-torus that is close to square but not exactly square. There is no isometry in the isotopy class of $f$.

\begin{proof}[Proof of \cref{lem:torsdiameter}]
    Fix $\delta > 0$. Suppose for the sake of contradiction that there is a sequence of metric spaces $M_i$ with $$\lim_{i\to \infty} \frac{\diam(\widetilde{M_i})}{\diam(M_i) |\Ht{M_i}|^\delta}=\infty$$
    Let's try to understand the coarse geometry of $\widetilde{M_i}$. Let $p$ be a point in $M_i$, and let $p_1\dots p_n$ be its lifts to $\widetilde M_i$. Let $B_1,\dots,B_n$ be the balls of radius $10\diam(M_i)$ centred at $p_1,\dots,p_n$. Form a graph $G_i$ whose vertices are in correspondence with the points $p_1,\dots,p_n$, and with an edge between $p_j$ and $p_k$ if $B_j$ and $B_k$ intersect. In other words, $G_i$ is a Cayley graph for $\Ht {M_i}$ with respect to generators having length $\lesssim10 \diam(M_i)$. The graph $G_i$ has $|\Ht{M_i}|$ vertices, is vertex transitive, and has diameter $\diam(\widetilde M_i)/\diam(M_i)$ up to a constant factor.
    
    By \cref{thm:torusconvergence}, the graphs $G_i$ Gromov-Hausdorff converge (after rescaling lengths by a factor of $1/\diam(G_i)$) to a finite dimensional torus $\Tl$ with an invariant Finsler metric. The idea of the proof is that $\Ht{M_i}$ acts on $G_i$, so it nearly acts on $\Tl$ by isometries. An abelian subgroup of $\Isom(\Tl)$ acting nearly transitively must be a group of translations. But if $\Ht{M_i}$ acts on $\Tl$ by translations, then there is a larger abelian cover of $M_i$ corresponding with unrolling the torus. To make this argument precise, we need to make sense of ``nearly acts''.

    Let $\varepsilon$ be a small number, to be determined later. Using \cref{lem:isomembed}, for large enough $i$, there is a $\varepsilon$-isometry $\iota: G_i \to \Tl$. By \cref{lem:isomembed} combined with \cref{lem:affine}, every isometry $g$ of $G_i$ can be approximated by an affine $\varepsilon$-isometry $\rho(g)$ of $\Tl$. More precisely, we mean that means that $\iota \circ g$ $\varepsilon$-commutes with $\rho(g)\circ \iota$. If $g_1$ and $g_2$ commute, then $\rho(g_1)$ and $\rho(g_2)$ $\varepsilon$-commute. Now apply $\rho$ to the actions of elements of $\Ht{M_i}$ on $G_i$ to obtain the family $\rho(\Ht{M_i}):=\{\rho(g) | g\in \Ht{M_i}\}$ of affine $\varepsilon$-isometries of $\Tl$ which pairwise $\varepsilon$-commute.
    
    Now $\rho$ is not necessarily an action of $\Ht{M_i}$ on $\Tl$. However, the only affine maps which are $\varepsilon$-close to the identity are translations. Therefore, $\rho$ descends to a homomorphism  $$\overline{\rho}:\Ht{M_i}\to \Aut(\Tl)/\text{translations} \cong GL(n, \Z)$$ 
    
    There are only finitely many elements of $GL(k,\Z)$ which can be represented by $\varepsilon$-isometries. Therefore, the image of $\overline \rho$ is finite and $[\Ht{M_i}:\ker(\overline \rho)]$ must be bounded above independent of $i$. The $\varepsilon$-neighbourhood of $\iota G_i$ covers $\Tl$. Hence, there exists an orbit $\mathcal O$ of the $\ker(\overline \rho)$ action on $G_i$ such that the $\varepsilon$-neighbourhood of $\iota \mathcal O$ occupies $[\Ht{M_i}:\ker(\overline \rho)]^{-1}$ of the volume of $T^k$. It follows that the set of translations in the image of $\rho$ has an $\varepsilon$-neighbourhood of volume at least $[\Ht{M_i}:\ker(\overline \rho)]^{-1}$ of the volume of $T^k$. Taking $\varepsilon$ small and invoking \cref{lem:commute}, the only affine $\varepsilon$-isometries of $\Tl$ which $\varepsilon$-commute with such a large set of translations are themselves $\varepsilon$-translations. So in fact $\coker(\overline \rho)=0$ and every element of $\Ht{M_i}$ must act as a translation.

    Now let $\alpha$ be a nontrivial element in $H^1(\Tl)$. Pull back $\alpha$ to a nontrivial element in $H^1(\widetilde{M_i})$. Since every element of $\Ht{M_i}$ acts as a translation, it acts trivially on $\alpha$. So we can average $\alpha$ with all of its $\Ht{M_i}$-translates to obtain an $\Ht{M_i}$ invariant class in $H^1(\widetilde{M_i})$. This class descends to a nontrivial element in $H^1(M_i)$. So $M_i$ has infinite $H_1$. This is a contradiction.
\end{proof}

\begin{lemma}\label{lem:contractcover}
    Suppose $M$ is a sequence of rational homology 3-spheres with Riemannian metrics. Let $\widetilde{M}$ be the universal abelian cover of $M$.  Then $$\diam(\widetilde M) \frac 1 {h_1(M)}\frac 1 {h_2(M)} \geq c\vol(M)$$ for some universal constant $c$.
\end{lemma}
\begin{proof}
Although the diameter of $\widetilde M$ appears in the inequality above, note that the other terms are the Cheeger constants of $M$ rather than $\widetilde M$. The pattern of the proof is very similar to the proof of \cref{thm:2}. So we will give an abbreviated description of the argument focusing mainly on the new aspects. Let $\pi: \widetilde{M}\to M$ be the covering map. The idea is that if the inequality is violated, we can construct a chain homotopy with integer coefficients betwen $\pi$ and a constant map. This contrasts with \cref{thm:2} where we attempted to construct a chain homotopy with $\frac 1 N \Z$ coefficients between the identity map $M\to M$ and a constant map.

Choose a basepoint $p\in M$ and a lift $\widetilde p \in \widetilde M$. Let $X$ be a very fine triangulation of $M$, so that the area of any simplex is $\leq \varepsilon$. Let $\widetilde X$ be its lift to $\widetilde M$. We will define the chain homotopy $H$ inductively on the $i$-skeleta of $X$. For each vertex $v\in \widetilde X$, let $\gamma$ be a path from $v$ to $\widetilde p$ of length $\leq \diam(\widetilde M)$. Then define $Hv = \pi(\gamma)$. For each edge $uv$ in $\widetilde X$, the 1-cycle $H(u) - H(v) + \pi(uv)$ lifts to a closed cycle in $\widetilde M$, so it is trivial in $H_1(M,\Z)$. This 1-cycle has length at most $2\diam(\widetilde M) + O(\varepsilon)$. By \cref{lem:ratint}, this 1-cycle can be filled in $M$ with an integer 2-chain of size $2\diam(\widetilde M)\frac 1 {h_1(M)} + O(\varepsilon)$. Note that the 1-cycle to be filled lives in $M$, so the relevant Cheeger constant is $h_1(M)$ rather than $h_1(\widetilde M)$. We declare $H(uv)$ to be this 2-chain.

As in the proof of \cref{thm:1}, extend $H$ to the 2-skeleton. This is always possible because $H_2(M,\Z)$ is trivial for a rational homology sphere. For any triangle $t$, we find
$$\norm{Ht} \leq 10\diam(\widetilde M) \frac 1 {h_1(M)}\frac 1 {h_2(M)} +O(\varepsilon)$$ and for any tetrahedron $r$, we have $$\norm{H\partial r + r}_1 \leq 40\diam(\widetilde M) \frac 1 {h_1(M)}\frac 1 {h_2(M)} +O(\varepsilon).$$

As in the proof of \cref{thm:1}, if $\norm{H\partial r + r} < \vol(M)$, then $H\partial r + r$ cannot be a nontrivial multiple of the fundamental class. So $H\partial r + r=0$ and we can define $Hr=0$ to complete the chain homotopy. This is a contradiction. So $\norm{H\partial r+r} \geq \vol(M)$ for some tetrahedron $r$, and the desired inequality is proven.
\end{proof}

\begin{proof}[Proof of \cref{thm:superpolytorsion}]
    Combine \cref{lem:torsdiameter} and \cref{lem:contractcover}, setting $\delta = \frac 1 k$.
\end{proof}

\begin{theorem}\label{thm:width}
    Suppose $M$ is a closed oriented rational homology 3-sphere with 1-bounded geometry. Suppose $f:M\to \R^2$ is a smooth map. Then there is a point $p\in \R^2$ such that $$|f^{-1}(p)|\geq c\frac 1 {h_1(M)} \frac 1 {h_2(M)}.$$ Here, $c$ is a universal constant.
\end{theorem}
\begin{proof}
    The proof is a standard application of Gromov's overlap theorem, with a few extra things to check since we are working with rational filling constants. 
    
    %Let $\mathcal Z(M,\Z)$ be the space of 1-cycles with $\Z$ coefficients in $M$ equipped with the flat metric. The flat distance between two homologous 1-cycles $a$ and $b$ is the minimum area of a 2-current in $M$ whose boundary is $a-b$. We call an element of $C_i(\mathcal Z(M,\Z))$ an \emph{$i$-sweepout}. (todo: what is the correct terminology?) It is important that we work with $\Z$ coefficients as opposed to $\R$ coefficients, because $\mathcal Z(M,\R)$ is contractible. If $f$ is sufficiently generic, it gives rise to a map $f^{-1}:\R^2 \to \mathcal Z(M,\Z)$ which sends a point $p$ to its preimage in $M$. The map $f^{-1}$ sends points near infinity to the trivial 1-cycle. In fact, $f^{-1}$ represents a 2-cycle in $\mathcal Z(M,\Z)$. This 2-cycle has a nontrivial homology class; one way to see this is to use the existence of a map $\phi:C_i(\mathcal Z(M,\Z)) \to C_{i+1}(M)$ which assembles together all of the 1-cycles in a chain of $\mathcal Z(M,\Z)$ into an $i+1$-cycle of $M$. In our case, $\phi(f^{-1})$ is the fundamental class of $M$, which is nontrivial. Therefore, $f^{-1}$ gives a nontrivial sweepout. There is also a disassembly map can convert some $i+1$-chains in $M$ to $i$-chains in $\mathcal Z(M,\Z)$. More precisely, given any $i+1$-chain $c$ along with a $\phi$-lift of $\partial c$ to $C_{i-1}(\mathcal Z(M,\Z))$,

    Choose a very fine triangulation $T$ of $\R^2$. For each vertex $v$ of $T$, $f^{-1}(v)$ is a 1-cycle in $M$. Crucially, this 1-cycle is trivial in $H_1(M,\Z)$, not just in $H_1(M,\R)$. Let $Hv\in C_2(M)$ be the least area integer filling of this 1-cycle. By \cref{lem:ratint}, it is also least area among rational fillings.

    For each edge $e$ of $T$, $H\partial e + f^{-1}(e)$ is an integer closed 2-cycle. Since $M$ is a rational homology sphere, this 2-cycle has a filling. By \cref{lem:ratint2}, the minimum volume rational filling of this 2-cycle may be taken to be an integer filling.
    
    Finally, for each triangle $t$ of $T$, $H \partial t + f^{-1}(t)$ is a closed 3-cycle. Let's work out how much control we have on $\norm{H\partial t + f^{-1}(t)}$. Let $\ell=\sup_{p\in \R^2} \norm{f^{-1}(p)}$. For any vertex $v$, edge $e$, or triangle $t$, we have
    \begin{align}
        \norm{Hv} &\lesssim \frac 1 {h_1(M)}\ell \nonumber\\
        \norm{He} &\lesssim \frac 1 {h_1(M)} \frac 1 {h_2(M)}\ell \nonumber\\
        \norm{H\partial t+f^{-1}(t)} &\lesssim \frac 1 {h_1(M)} \frac 1 {h_2(M)}\ell\label{eqn:x1}
    \end{align}

    On the other hand,
    \begin{align*}
    \sum_{t\in T} H\partial t + f^{-1}(t)&= \sum_{t\in T} f^{-1}(t)\\
    &=[M]
    \end{align*}
    Therefore, there exists at least one triangle $t_0$ such that $H\partial t_0 + f^{-1}(t_0)$ is a nontrivial multiple of the fundamental class of $M$. So
    \begin{align} \label{eqn:x2}
    \norm{H\partial t_0+f^{-1}(t_0)} &\geq \vol(M).
    \end{align}
    Combining \cref{eqn:x1} and \cref{eqn:x2}, we get
    \begin{align*}
    \ell \gtrsim h_1(M)h_2(M)\vol(M)
    \end{align*}
    
\end{proof}
\begin{proof}[Proof of \cref{thm:3}]
Apply \cref{thm:width} to the sequence of manifolds $M_i$ from \cref{thm:0}, which have $1/h_1(M)$ and $1/h_2(M)$ both $\lesssim \poly \log(\vol(M))$.
\end{proof}

\begin{section}{Questions}\label{sec:questions}
\begin{question}
    Do there exist sequences of $n$-manifolds with unbounded volume and $1/h_i(M)=o(\vol(M)^\varepsilon)$ for $i=0\dots n-1$?
\end{question}
\begin{question}
 Can one produce examples of hyperbolic 3-manifolds with uniform lower bounds on the spectral gap for the  Hodge Laplacian on all $i$-forms?
 \end{question}
 One approach would be to start with a tower of covers of 2-complexes that are good spectral expanders. Our construction respects covers, so it would be enough to arrange that 3-manifold at the bottom of the tower is hyperbolic.
    \begin{question}
      Let $X$ be a metric space with $H_1(X,\Z)$ finite. Let $\widetilde X$ be its universal abelian cover. \cref{lem:torsdiameter} gives an upper bound on $\diam(\widetilde X)/\diam(X)$ in terms of $|H_1(X)|$. Could this bound be improved to $$\diam(\widetilde X)/\diam(X) \lesssim \log(|H_1(X)|)?$$
      \end{question}
      If true, this would improve the superpolynomial lower bound on torsion homology in \cref{thm:superpolytorsion} to an exponential lower bound. Currently, our proof of \cref{lem:torsdiameter} uses the results of \cite{benjamini_scaling_2014} to constrain the Gromov-Hausdorff limit of $\widetilde X$. Their results work for arbitrary vertex transitive graphs. One might be able to prove better results when the symmetry group of the graph is abelian as in our case.
 
    \begin{question}{\cite[Naive conjecture 4]{guth_metaphors_2010}}
        If $g$ is a metric on $T^3$ of 1-bounded geometry, does there exist a continuous map $f:T^3 \to \R^2$ so that for every $p\in \R^2$, the length of the fiber $f^{-1}(p)$ is controlled by the volume of $g$
        $$length(f^{-1}(p)) \leq C \vol(T^3)^{1/3}?$$
    \end{question}
    Guth promoted this question as one of the simplest open questions about metric geometry of 3-manifolds. We do not know any upper bound better than the trivial $O(\vol(T^3))$. The question is also open for any fixed 3-manifold. The technique we used to show that our manifolds have large width requires good rational expansion; therefore, \cref{thm:superpolytorsion} indicates that this strategy cannot prove that a sequence of Riemannian 3-manifolds with bounded torsion homology has large width over $\R^2$.
\end{section}

\printbibliography
\end{document}